\documentclass[12pt]{amsart}
\usepackage{amsthm,amsmath,amssymb,latexsym,graphics}
\usepackage{hyperref}
\newtheorem{proposition}{Proposition}
\newtheorem{theorem}{Theorem}
\newtheorem{lemma}{Lemma}
\newtheorem*{corollary}{Corollary}
\newtheorem{claim}{Claim}
\newtheorem*{thm}{Theorem 1$'$}
\newtheorem*{thmm}{Theorem 1$''$}
\newtheorem*{thmA}{Theorem A}
\newtheorem*{thmB}{Theorem  B}
\newtheorem*{thmC}{Theorem C}
\newtheorem*{thmD}{Theorem D}
\newtheorem*{thmE}{Theorem E}
\newtheorem*{remark}{Remark}

\newcommand{\grad}{\nabla}

\newcommand{\mint}{\backslash\!\!\!\!\!\int_{\mathbb S^{n-1}}}
\newcommand{\wph}{\widehat{\phi}(t,\theta)}

\begin{document}
\title[Asymptotic  behavior  of  solutions  to the $\sigma_k$-Yamabe equation]{Asymptotic behavior  of  solutions  
to the $\sigma_k$-Yamabe equation near isolated singularities}
\author{Zheng-Chao Han }
\address{Department of Mathematics \\Rutgers University \\110 Frelinghuysen Road \\
Piscataway, NJ 08854}
\email{zchan@math.rutgers.edu} 
\author{YanYan Li}
\email{yyli@math.rutgers.edu}
\author{Eduardo V. Teixeira}
\thanks{The research of the second author was supported in part by DMS-0701545, and
the research of the third author was  supported in part by NSF
grant DMS-0600930 and CNPq-Brazil.}
\curraddr[Eduardo V. Teixeira]{Universidade Federal do Cear\'a \\
Departamento de Matem\'atica \\
Av. Humberto Monte, s/n \\
Fortaleza-CE, Brazil.  CEP 60.455-760}
\email{eteixeira@ufc.br}
\date{}
\begin{abstract}
$\sigma_k$-Yamabe equations are conformally invariant equations generalizing the classical  Yamabe equation.
In \cite{Li06} YanYan Li proved that an admissible solution with an isolated singularity at 
$0\in \mathbb R^n$ to the $\sigma_k$-Yamabe equation is asymptotically radially symmetric.
In this work we prove that an admissible solution with an isolated singularity at 
$0\in \mathbb R^n$ to the $\sigma_k$-Yamabe equation is asymptotic to a  radial  solution to the same
equation on $\mathbb R^n \setminus \{0\}$. 
These results generalize earlier pioneering work in this direction on the classical  Yamabe equation by
Caffarelli, Gidas, and Spruck.  In extending the work of Caffarelli et al,
we  formulate and prove a general asymptotic approximation
result for solutions to certain ODEs which include the case for scalar curvature and $\sigma_k$ curvature cases.
An alternative proof is also provided using analysis of the linearized operators at the  radial
solutions, along the lines of approach in a work by Korevaar, Mazzeo, Pacard, and Schoen.
\end{abstract}

\maketitle
\section{Description of the results}
In a classic paper  \cite{CGS} Caffarelli, Gidas, and Spruck proved the asymptotic radial symmetry of 
positive singular solutions $u$ to the conformal scalar curvature equation
\begin{equation}\label{sca}
\Delta u(x) + \frac{n(n-2)}{4} u^{\frac{n+2}{n-2}}(x) =0
\end{equation}
on a punctured ball, and further proved that such solutions are asymptotic to   radial
singular solutions to \eqref{sca} on $\mathbb R^n \setminus \{0\}$. To describe the results of
\cite{CGS} more precisely, we first describe the radial solutions to \eqref{sca}
on $B_R(0)\setminus\{0\}$ for $0< R \le \infty$. A positive solution $u$ to \eqref{sca} corresponds to a 
conformal metric 
$$ g= u^{\frac{4}{n-2}}(x) |dx|^2$$
 with scalar curvature $n(n-1)$. Using the polar coordinates
$x=r\theta$, with $r=|x|$, and $\theta \in \mathbb S^{n-1}$, we can introduce cylindrical variable $t=-\ln r$, so
that
\begin{equation} \label{trans}
g= u^{\frac{4}{n-2}}(x) |dx|^2 =  %u^{\frac{4}{n-2}}(x) \left( dr^2+r^2 d\theta^2 \right) =
\left[r^{\frac{n-2}{2}} u(r\theta)\right]^{\frac{4}{n-2}}  \left( r^{-2} dr^2+ d\theta^2 \right)
= U^{\frac{4}{n-2}}(t, \theta) \left( dt^2 + d\theta^2 \right),
\end{equation}
where $U(t, \theta) = r^{\frac{n-2}{2}} u(r\theta)$. Computing the scalar curvature of $g$ in terms
of $U$ and the background cylindrical metric $ dt^2 + d\theta^2$, we can transform \eqref{sca} into
\begin{equation}\label{cyl}
U_{tt}(t,\theta) + \Delta_{ \mathbb S^{n-1}} U(t,\theta) - \frac{(n-2)^2}{4} U(t,\theta) + \frac{n(n-2)}{4} U^{\frac{n+2}{n-2}} 
(t,\theta) =0.
\end{equation}
If $u(x)=u(|x|)$ is a radial positive solution to \eqref{sca} in some $B_R(0)\setminus\{0\}$, then
$ \psi (t):=U(t,\theta)= r^{\frac{n-2}{2}} u(r\theta)$ is a positive  solution to the ODE
\begin{equation} \label{cylrad}
\psi _{tt}(t) - \frac{(n-2)^2}{4} \psi (t) + \frac{n(n-2)}{4} \psi ^{\frac{n+2}{n-2}}(t) =0,
\end{equation}
for $t > - \ln R$.  Solutions to \eqref{cylrad} has a first integral:
\begin{equation*}%\label{1stint}
H := \psi _t^2(t) + \frac{(n-2)^2}{4} \left[  \psi ^{\frac{2n}{n-2}}(t) - \psi ^2(t) \right] \equiv \text{const.}
\end{equation*}
along any  positive  solution  $\psi (t)$. 
In fact, positive solutions $U$ to \eqref{cyl} globally defined on the entire cylinder $\mathbb R \times \mathbb S^{n-1}$
(thus positive solutions $u$ to \eqref{sca} on $\mathbb R^n \setminus \{0 \}$)
are classified in \cite{CGS}.
\begin{thmA} (\cite{CGS})
 Let $U(t,\theta)$ be any positive solution to \eqref{cyl} defined on the entire
cylinder $\mathbb R \times \mathbb S^{n-1}$. If $0$ is a non-removable singularity of $u$ in the sense that
\[
\liminf_{t\to \infty} \min_{\theta \in  \mathbb S^{n-1}} U(t,\theta) >0,
\]
then $U$ is independent of $\theta$. Moreover $U(t)$ is a periodic
solution  of \eqref{cylrad}  with $0< U(t) \le 1$ for all
$t \in \mathbb R$ and the first integral $H < 0$.
We refer to these solutions as global singular positive solutions to \eqref{cyl}.
If $0$ is a removable singularity of $u$ in the sense that 
\[
\liminf_{t\to \infty}\min_{\theta \in  \mathbb S^{n-1}}  U(t,\theta) =0,
\]
then the corresponding $u(x)= e^{\frac{n-2}{2} t} U(t,\theta)= |x|^{- \frac{n-2}{2}}U(-\ln |x|, \frac{x}{|x|})$ 
is identically equal to
$\displaystyle{
\left( \frac {2a}{ 1+a^2|x-\bar x|^2 }\right)^{ \frac {n-2}2 }
}
$
in $\mathbb R^n$ for some
$\bar x\in \mathbb R^n$ and
$a>0$.
\end{thmA}
Thus global singular positive solutions $U(t,\theta)$ to \eqref{cyl} defined on the entire
cylinder $\mathbb R \times \mathbb S^{n-1}$ can be parametrized by two parameters:
 its minimum value $\epsilon>0$ and
a moment $T$ when it attains this minimum value. 
It turns out that the minimum value $\epsilon>0$ of global singular  
positive solutions  to \eqref{cyl}  has the restriction $0 < \epsilon \le \epsilon_0$, where $\epsilon_0 = \left( 
\frac{n-2}{n}\right)^{\frac{n-2}{4}}$. For any such  $\epsilon $, let $\psi_{ \epsilon }(t)$ denote the solution
to \eqref{cylrad}  such that $\psi_{ \epsilon }(0)= \epsilon$ and $\psi_{ \epsilon }^{'}(0)=0$. Then any
global singular positive solution  to \eqref{cyl} can be represented as $\psi_{\epsilon }(t+\tau)$ for some 
$0 < \epsilon \le \epsilon_0$ and $\tau$.

The main results in  \cite{CGS} on the asymptotic behavior of a positive solution to \eqref{sca} in $B_R\setminus \{0\}$
can be stated as 
\begin{thmB}(\cite{CGS})
Suppose that $u(x)$ is a positive solution to \eqref{sca} in $B_R\setminus \{0\}$ and does not extend to a smooth
solution to \eqref{sca}  over $0$, then 
\begin{equation}\label{ave}
u(x) = \bar u(|x|) \left( 1+O(|x|)\right) \quad \text{as $x \to 0$},
\end{equation}
with
\[
\bar u(|x|) = \mint u(|x|\theta) \, d \theta
\]
being the spherical average of $u$ over the sphere $\partial B_{|x|}(0)$; 
furthermore,
there exists a radial singular solution $u^*(|x|)$ to  \eqref{sca} on $\mathbb R^n \setminus \{0\}$ and
some $\alpha>0$, $0 < \epsilon \le \epsilon_0$ and $\tau$ such that
\begin{equation}\label{radsymm}
u(x) = u^*(|x|) \left(1+ O(|x|^{\alpha})\right) \quad \text{ as $|x| \to 0$.}
\end{equation}
and
\begin{equation*}%\label{radsymmr}
 u^*(|x|)=|x|^{-\frac{n-2}{2}} \psi_{\epsilon }(-\ln |x| +\tau) \quad \text{ as $|x| \to 0$.}
\end{equation*}
\end{thmB}
A key ingredient in the  proof in \cite{CGS} of Theorem B  uses a 
``measure theoretic" variation of the
moving plane technique, which had been developed  by Alexandrov \cite{A5?}, Serrin \cite{S71}, 
and Gidas-Ni-Nirenberg \cite{GNN} to prove symmetries of solutions to certain elliptic PDEs. 
Subsequent to \cite{CGS} there have been many papers related to the theme of 
 Theorem B, including
 \cite{CL95}, \cite{CL96},  \cite{KMPS}, and \cite{TZ06}, among others.
In particular  \cite{KMPS} gives a proof of \eqref{radsymm} and 
provides an  expansion of $u$ after the order $u^*(|x|)$
using rescaling analysis, classification of global
singular solutions as given by Theorem A, and analysis of linearized operators at these global 
singular solutions. 

Our objective in this paper is to study similar problems for
singular solutions in a punctured ball to a family of conformally invariant equations
which  include \eqref{sca}.
More specifically, we consider singular solutions to the equation 
\begin{equation} \label{1}
\sigma_k (g^{-1} \circ A_g) = \text{constant},
\end{equation}
on a punctured ball   $\{\,x\in \mathbb R^n : 0 < |x| < R \,\}$,
 where 
\[
A_g =  \frac {1}{n-2} \{ Ric_g - \frac {R_g}{2(n-1)}g \},
\]
is the Weyl-Schouten tensor of the conformal metric 
$$g = u^{\frac{4}{n-2}}(x) |dx|^2,$$ 
$Ric_g$ and $R_g$ denote
respectively the Ricci and scalar curvature of $g$,
and $\sigma_k(g^{-1} \circ A_g)$ denotes the $k$-th elementary symmetric function of the eigenvalues of $A_g$
with respect to $g$, and $0  < R \le \infty$.  Due to the transformation law,
see e.g. \cite{Via1},
\[
A_g = A_{g_0} + \left[ \nabla^2w+dw \otimes dw - \frac 12 |\nabla w|^2 g_0\right],
\]
when $g=e^{-2w}g_0$,
and
\[
\sigma_k(g^{-1} \circ A_g) = e^{2kw} \sigma_k(g_0^{-1} \circ A_g),
\]
\eqref{1} is equivalent to 
\begin{equation} \label{1w}
%\begin{split}
%& 
\sigma_k(g_0^{-1} \circ \{A_{g_0} +  \nabla^2w+dw \otimes dw - 
\frac 12 |\nabla w|^2 g_0   \}) 
%\\
= 
%& 
c e^{-2k w},
%\end{split}
\end{equation}
for some constant $c$, where we will often
take $g_0$ to be the round cylindrical metric $dt^2+d\theta^2$ on $\mathbb R \times \mathbb S^{n-1}$, 
which is pointwise conformal to the flat metric $|dx|^2$ on $\mathbb R^n$ as seen through \eqref{trans}.
For ease of reference, we also record an equivalent formulation of \eqref{1w} in  terms of  $u(x)$ through the transformation
\eqref{trans}:
\begin{equation}\tag{\ref{1w}$'$}\label{1w'} 
%\begin{split}
%& 
\sigma_k\left( -(n-2) u(x)
 \grad^2 u(x) + n \grad u(x) \otimes \grad u(x) -|\grad u(x)|^2 
\text{Id}\right) 
%\\
=
%&
\frac{2^k c}{(n-2)^{2k}}  u^{\frac{2kn}{n-2}}(x).
%\end{split}
\end{equation}

The  study of singular solutions
of equations of the above type
 is related to the characterization of the size of 
the limit set of the image domain in $\mathbb S^n$ of
the developing map of a locally conformally flat $n$-manifold. More specifically, one is led to  looking   for
necessary/sufficient conditions  on a domain $\Omega \subset
\mathbb S^n$ so that it 
admits a  metric $g$  which is pointwise conformal to the standard metric on $\mathbb S^n$,  complete, 
and with its Weyl-Schouten
tensor $A_g$ in the $\Gamma^{\pm}_k$ class, \emph{i.e.}, the eigenvalues, 
$\lambda_1 \le \cdots \le \lambda_n$, of $A_g$ at each $x \in  \Omega$ satisfy
$ \sigma_j (\lambda_1 , \cdots , \lambda_n) >0$ for all $j, 1\le j \le k$, 
in the case of 
$\Gamma^+_k$; and $(-1)^j \sigma_j (\lambda_1 , \cdots , \lambda_n) >0$ 
for all $j, 1\le j \le k$, in the case of $\Gamma^-_{k}$.  For $k \ge
2$, 
it is often 
restricted
 to metrics whose Weyl-Schouten tensor is in  the $\Gamma^{\pm}_k$ class,
 because, for a metric in such a class, \eqref{1w} becomes a  
fully nonlinear PDE in $w$ that is 
\emph{elliptic}.
 In the case of $k=1$, $\sigma_1(A_g)$ is simply a  positive constant multiple of the scalar curvature of $g$; so
 $A_g$ in the $\Gamma^{\pm}_k$ class is a generalization of the notion that the scalar
curvature $R_g$ of $g$ having a fixed $\pm$ sign. For the positive scalar curvature case, 
Schoen and Yau proved in \cite{SY} 
 that if a complete conformal metric $g$
exists on a domain $\Omega \subset \mathbb S^n$ with $\sigma_1 (A_g)$ having a positive lower bound, 
then the Hausdorff dimension of $\partial \Omega$ has to be $\leq (n-2)/2$. In \cite{S88} Schoen constructed
complete conformal metrics on $\mathbb S^n \setminus \Lambda$ when $\Lambda$ is either a finite discrete set on
$\mathbb S^n$ containing at least two points or a set arising as the limit set of a Kleinian group action. Later 
Mazzeo and Pacard \cite{MP1} proved that if $\Omega \subset \mathbb S^n$ is a domain such that
$\mathbb S^n \setminus \Omega$ consists a finite number of
disjoint smooth submanifolds of dimension 
$1\le k \le (n-2)/2 $,
then one can find a complete conformal metric $g$ on  $ \Omega$ with its scalar
curvature  identical to $+1$. For the negative scalar curvature case, the results of
 Loewner-Nirenberg \cite{LN}, Aviles \cite{Av}, and Veron \cite{Ver} imply that if $\Omega \subset \mathbb S^n$
admits a complete, conformal metric with negative constant scalar curvature, then the Hausdorff
dimension of $\partial \Omega > (n-2)/2$.   Loewner-Nirenberg \cite{LN} also proved that  if 
$\Omega \subset \mathbb S^n$ is a domain with smooth boundary $\partial \Omega$ of dimension $> (n-2)/2$,
then there exists a  complete conformal metric $g$
on  $\Omega$ with $\sigma_1 (A_g) = -1$. This result was later generalized by D. Finn \cite{Fin} to the case
of $\partial \Omega$ consisting of smooth submanifolds of dimension $> (n-2)/2$ and  with boundary. 
For more recent development  related to the  negative scalar curvature case, 
see \cite{Lab03}, \cite{leG}, \cite{MV09}  and the references therein.
The consideration of singular solutions of equations of type \eqref{1w} can be considered as a natural 
generalization of these known results. In fact, in \cite{CHgY}, Chang, Hang, and Yang proved that if
$\Omega \subset \mathbb S^n$ ($n \geq 5$) admits a complete, conformal metric $g$ with 
$$\sigma_1(A_g) \geq c_1 >0, \quad  \sigma_2(A_g) \geq 0, \quad \text{and}$$
\begin{equation} \label{ub}
|R_g| +|\grad_g R|_g \leq c_0, 
\end{equation}
then $\dim (\mathbb S^n \setminus \Omega) < (n-4)/2$. This has been generalized by M. Gonzalez \cite{G2}
and Guan, Lin and
Wang \cite{GLW} to the
case of $2 < k < n/2$: if
$\Omega \subset \mathbb S^n$  admits a complete, conformal metric $g$ with 
$$\sigma_1(A_g) \geq c_1 >0, \quad  \sigma_2(A_g),\;\cdots\;, \sigma_k(A_g)  \geq 0, \quad \text{and}$$
\eqref{ub}, then $\dim (\mathbb S^n \setminus \Omega) < (n-2k)/2$.

We restrict our attention in this paper to singular solutions with isolated singularity. 
We say a solution to \eqref{1w} or \eqref{1w'} %(\ref{1w}\,$'$)
is in the $\Gamma^+_k$ class in some region if its associated Weyl-Schouten tensor is in the $\Gamma^+_k$ there;
for a positive function $u$  to \eqref{1w'}, this means that  the matrix 
$ \left(-(n-2) u(x)
 \grad^2 u(x) + n \grad u(x) \otimes \grad u(x) -|\grad u(x)|^2
\text{Id}\right)
$ belong to $\Gamma^+_k$.
Our main result is 
\begin{theorem}\label{main}
Let $w(t,\theta)$ be a smooth
 solution to \eqref{1w} on
$\{t>t_0\}\times \mathbb S^{n-1}$ in the $\Gamma^+_k$ class, where  $n\ge 3$,
 $2\le k \le n$, and the constant $c>0$.
  Then there exist a \underline{radial} solution $w^*(t)$ 
to \eqref{1w} on $\mathbb R \times \mathbb S^{n-1}$ in the $\Gamma^+_k$ class,
and constants $\alpha>0$, $C>0$ such that
\begin{equation}\label{maincon0}
|w(t,\theta)-w^*(t)| \le C  e^{-\alpha t}\quad \text{for $t>t_0+1$}.
\end{equation}
%$\alpha>0$, $h\ge 0$ and $\tau$ such that
%\begin{equation}\label{maincon}
%|w(t,\theta)-\xi_h(t+\tau)|\le C e^{-\alpha t}.
%\end{equation}
\end{theorem}
We can formulate Theorem~\ref{main} in terms of the variable $u(x)$ defined on 
$B_R\setminus\{0\}$ through \eqref{trans}:
\begin{thm}
Let  $u(x)$ be a positive
smooth
 solution to \eqref{1w'} on $B_R \setminus \{0\}$ in the $\Gamma^+_k$ class,
 where $n\ge 3$, $2\le k \le n$,
$R>0$,  and $c>0$, 
then there exist 
a positive radial smooth
 solution $u^*(|x|)$ to \eqref{1w'} on $\mathbb R^n
\setminus \{0\}$ in the $\Gamma^+_k$ class, and  constants $\alpha>0$, $C>0$ such that
\begin{equation}\label{mainconu}
|u(x)-u^*(|x|)| \le C|x|^{\alpha} u^*(|x|)\quad \text{for $|x|< R/2$}.
\end{equation}
\end{thm}

\begin{remark}
As a consequence of 
Theorem \ref{hodthm} below, the $\alpha$ in 
Theorem ~\ref{main} and ~\ref{main}$'$
can be any number in $(0, 1)$, while the constant
$C$ depends also on $\alpha$.
\end{remark}

\begin{remark}
The  $k=1$ case of Theorems~\ref{main} and ~\ref{main}$'$ was proved in \cite{CGS}, as remarked earlier.
In the situation of  Theorem~\ref{main} $'$, the asymptotic symmetry of $u(x)$, a positive  solution to \eqref{1w'} on 
$B_R \setminus \{0\}$ in the $\Gamma^+_k$ class, was
proved in \cite{Li06}, namely, for some constant $C>0$, 
\[
|u(x)-\bar u(|x|)| \le C|x| \bar u(|x|),
\]
where $\bar u(|x|)$ is the spherical average of $u(x)$. This is a generalization of the result
\eqref{ave} in \cite{CGS} for solutions to \eqref{sca}, and will be a starting point for our proof of
Theorem~\ref{main}.
\end{remark}
We can describe the asymptotic behavior of solutions in
Theorems~\ref{main} and \ref{main}$'$ in more explicit terms after describing the classification
results from \cite{CHY05} on the radial solutions to \eqref{1w} for $k>1$ in the $\Gamma^+_k$ class
globally defined on $\mathbb R \times \mathbb S^{n-1}$.
Let us first work out \eqref{1w} more explicitly in the case of radial
solutions.

To fix our notations,  we introduce    
new  variables  $v(t,\theta)$ and $w(t,\theta)$ such that
\begin{equation} \tag{\ref{trans}$'$}\label{trans'}
g= u^{\frac{4}{n-2}}(x)|dx|^2 = v^{-2}(x) |dx|^2= U^{\frac{4}{n-2}}(t,\theta)(dt^2 +d \theta^2)
= e^{-2w(t,\theta)} (dt^2 +d \theta^2),
\end{equation}
where $t = - \ln |x|$ and $\theta = x/|x|$. Thus 
\begin{equation}\label{change}
|x|^{\frac{n-2}{2}}u(x) =\left(\frac{|x|}{v(x)}\right)^{\frac{n-2}{2}}= U(t,\theta)=e^{-\frac{n-2}{2}w (t,\theta)}.
\end{equation}
%If we write the metric $g$ as  
%$$g =e^{-2w(t,\theta)}(dt^2+d\theta^2)= v^{-2}(|x|) |dx|^2, \quad \text{and} \quad |x| = r,$$
Following the notation in \cite{CHY05},  the Schouten tensor of $g$ 
can be computed as, when $v=v(|x|)$,
$$
%\begin{eqnarray*}
A_{ij}  
%&
= \frac {v_{ij}}{v} - \frac {|\grad v|^2}{2v^2} \delta_{ij} 
%\\
        %&
= \lambda \delta_{ij} + \mu \frac{x_ix_j}{|x|^2},
%\end{eqnarray*}
$$
with $\lambda = \frac{v_r}{rv} (1 - \frac{rv_r}{2v})$ and 
 $\mu = \frac{v_{rr}}{v} -\frac{v_r}{rv}$.
The eigenvalues of $A$ with respect to $|dx|^2$ are $\lambda$ with
 multiplicity $(n-1)$, and 
$\lambda +\mu$ with multiplicity $1$. The formula for $\sigma_k(g^{-1}\circ A_g)$ can
be found easily by the binomial expansion of $(x-\lambda)^{n-1}(x-\lambda-\mu)$:
\begin{equation} \label{2}
\sigma_k(g^{-1}\circ A_g) = c_{n,k} v^{2k} \lambda^{k-1} (n\lambda + k \mu),
\end{equation}
where $c_{n,k} = \frac {(n-1)!}{k! (n-k)!}=\frac 1n \binom{n}{k}$.

\begin{remark}
The convention $t = -\ln r$ here is off by a sign with the convention in \cite{CHY05}, and the transformation
from $v$ to $w$ is adjusted from \cite{CHY05} accordingly.
\end{remark}

\[v_r(x)=e^{w(t,\theta)}(-w_t(t,\theta)+1)=(-w_t(t,\theta)+1)v(x)/r, 
\]
and
\[
v_{rr}(x)= e^{w(t,\theta)+t}[w_{tt}(t,\theta)-w_t(t,\theta)(-w_t(t,\theta)+1)] 
= [w_{tt}(t,\theta)-w_t(t,\theta)(-w_t(t,\theta)+1)] v(x)e^{2t}.
\]
Thus when $v=v(|x|)$, we find $ w(t,\theta)= t+ \ln v(e^{-t}) =: \xi (t)$ is a function of $t$, and 
\[
\lambda = \frac 12 e^{2t} (1-\xi_t^2), \qquad \text{and} \qquad 
\mu = e^{2t}(\xi_{tt} + \xi_t^2 -1).
\]
Using \eqref{2}, \eqref{1w} in the radial case then becomes
\begin{align} \label{3}
c=\sigma_k(A_g) &= c_{n,k} e^{2k(\xi+t)}\frac{(1-\xi_t^2)^{k-1}}{2^{k-1} e^{2(k-1)t}}
\left[n \frac{1-\xi_t^2}{2e^{2t}} +  k\frac{\xi_{tt}+\xi_t^2-1}{e^{2t}}\right] \notag \\
&= c_{n,k}^{'} (1-\xi_t^2)^{k-1} \left[\frac{k}{n} \xi_{tt} + (\frac{1}{2} -\frac{k}{n})(1-\xi_t^2)\right]e^{2k\xi},
\end{align}
where $c_{n,k}^{'} = n c_{n,k} 2^{1-k}= 2^{1-k}\binom{n}{k} $.
Thus \eqref{3} is the radial case of \eqref{1}, written in cylindrical coordinate $t$ and the variable $w(t,\theta)=\xi(t)$.
In general, we will allow ourselves the flexibility of  treating \eqref{1} either as an equation 
for $u(x)$ on $B_R \setminus \{0\}$ or as an equation for $U(t,\theta)$ or $w(t,\theta)$ 
on a cylinder $\{(t,\theta): t> - \ln R, \theta \in 
\mathbb S^{n-1} \}$ with respect to the background metric $dt^2+d\theta^2$.

%Note that any solution $\xi (t)$ of \eqref{3}, with $k>1$ and  $\sigma_k$ is a \emph{positive} constant, has the property that
%either $1-\xi^2_t<0$, or $1-\xi^2_t>0$ along its entire interval of existence; $\xi (t)$ also has a first integral: 
%when  $\sigma_k$ is  normalized to be $2^{-k}\binom{n}{k}$, $e^{(2k-n)\xi (t)}(1-\xi_t^2 (t))^k - e^{-n\xi (t)}$ is a  constant. 
%Denote this constant by $h$.
 
%The behavior of the radial solutions of \eqref{3} in the $\Gamma^+_k$ class on the entire
%$\mathbb R \times \mathbb S^{n-1}$, when $k>1$ and 
% $\sigma_k$ is a \emph{positive} constant, normalized to be $2^{-k}\binom{n}{k}$, 
%are cataloged in \cite{CHY} as follows. 
Now we record the relevant part of the results in  \cite{CHY05} regarding 
radial solutions of \eqref{3} in the $\Gamma^+_k$ class on the entire $\mathbb R \times \mathbb S^{n-1}$, 
when $k>1$ and
$\sigma_k$ is a \emph{positive} constant, normalized to be $2^{-k}\binom{n}{k}$.
\begin{thmC} (\cite{CHY05})
Any  radial solution $\xi (t) :=w(t,\theta)$ of \eqref{1w} in the $\Gamma^+_k$ class on the \underline{entire}
$\mathbb R \times \mathbb S^{n-1}$, when $k>1$ and 
 $c$ is a \emph{positive} constant, normalized to be $2^{-k}\binom{n}{k}$, has the property that
$1-\xi_t^2 > 0$
for all $t$. Furthermore, $h:= e^{(2k-n)\xi (t)}(1-\xi_t^2 (t))^k - e^{-n\xi (t)}$ is a  nonnegative
constant.  Moreover 
  \begin{enumerate}
	\item  If $h=0$, then 
	   $u^{\frac{4}{n-2}}(|x|)= \left(\frac{2\rho}{|x|^2+\rho^2}\right)^2$ for some positive parameter
	$\rho$. So these solutions  give rise to the round
	  spherical metric on $\mathbb R^n \cup \{ \infty \}= \mathbb S^n$.
	\item If $h>0$, then the behavior of $u$ is classified according to the relation between 
	      $2k$ and $n$:
		\begin{enumerate} 
			\item If $2k < n$, then $h$ has the further restriction
			$h \leq  h^* :=\frac{2k}{n-2k}\left( \frac{n-2k}{n} \right)^\frac{n}{2k}$ and
          			 $\xi (t)$ is a periodic
          			function of $t$, giving rise to a metric 
				$g = \frac{e^{-2\xi(\ln |x|)}}{|x|^2} |dx|^2 $ on
          			$\mathbb R^n \setminus \{0\}$ which is complete. Note that the case $h=h^*$
				gives rise to the cylindrical metric $\frac{|dx|^2}{|x|^2}$ on 
				$\mathbb R^n \setminus \{ 0\}$.
			\item If  $2k = n$, then $h$ satisfies the further restriction $h<1$ and
          			as $|x| \to 0$, $g=u^{\frac{4}{n-2}}(|x|)|dx|^2$ has the asymptotic
				$$g \sim |x|^{-2(1- \sqrt{1-\sqrt[k]{h}})} |dx|^2
                             =e^{-\left(2 \sqrt{1-\sqrt[k]{h}}\right)t}
                            (dt^2+d\theta^2),$$ 
                                and as $|x| \to \infty$,
				$g=u^{\frac{4}{n-2}}(|x|)|dx|^2$ has the asymptotic 
				$$g \sim |x|^{-2(1+ \sqrt{1-\sqrt[k]{h}})} |dx|^2=e^{ 2 \sqrt{1-\sqrt[k]{h}}t}
                                   (dt^2+d\theta^2).$$ 
                                Thus  
				$g$ gives rise to a metric on  $\mathbb R^n
          			\setminus \{0\}$ singular at $0$ and at $\infty$ which
          			behaves like the cone metric, is incomplete with
          			finite volume.
			\item If $2k > n$, then $u^{\frac{4}{n-2}}(|x|)$ has an asymptotic 
				expansion of the form 
		\[
		u^{\frac{4}{n-2}}(|x|) = \rho^{-2}\{1- \sqrt[k]{h} \frac{k}{2k-n} 
				\left( \frac{|x|}{\rho}\right)^{2-\frac{n}{k}}+ \cdots \}
		\] 
				as $|x| \to 0$, where $\rho>0$ is a positive parameter,
				thus $u(|x|)$ has a positive, finite limit,  but $u_{rr}(|x|)$
				 blows up at $|x| \to 0$. The behavior of $u$
				as $|x|\to \infty$ can be described similarly. Putting together, 
				we conclude that
				$u^{\frac{4}{n-2}}(|x|)|dx|^2$ extends to a $C^{2-\frac{n}{k}}$ metric on $\mathbb S^n$.
                \end{enumerate}
\end{enumerate}
\end{thmC}

We can parametrize the global singular radial solutions
to \eqref{1w} in a similar way as before: for each $ 0 < h $, subject to any further constraints depending on
$2k<$ or $=n$, as given in Theorem C,
let $\xi_{h}(t)$ denote the solution to  \eqref{1w} with its first integral equal to $h$
and such that $\xi_h(0)$ equals $\min_{\mathbb R} \xi_h(t)$.
We can now reformulate Theorem~\ref{main} with more explicit information as

\begin{thmm}
Let $w(t,\theta)$ be a smooth solution to \eqref{1w} on
$\{t>t_0\}\times \mathbb S^{n-1}$ in the $\Gamma^+_k$ class, where $n\ge 3$, $2\le k \le n$, and
$c>0$.
Then there exist $\alpha>0$, $h\ge 0$, $\tau$ and $C>0$  such that
\begin{equation}\label{maincon}
|w(t,\theta)-\xi_h(t+\tau)|\le C e^{-\alpha t}\quad \text{for $t>t_0+1$}.
\end{equation}
\end{thmm}
As mentioned earlier, the $\alpha$ in the above theorem can be taken as
any number in $(0, 1)$, while the constant $C$ then depends on
$\alpha$ as well.

Using the transformation \eqref{trans'} and our knowledge of $\xi_h(t+\tau)$ as given in Theorem C,
we can also formulate the result in Theorem~\ref{main}$''$ in terms of the 
variable $u(x)$ defined on $B_R\setminus\{0\}$.
\begin{corollary}
Let $u(x)$ be a positive
smooth  solution to \eqref{1w'} on $B_R \setminus \{0\}$ 
in the $\Gamma^+_k$ class, where the constant
$c$ is normalized to be  $2^{-k}\binom{n}{k}$.  If
\begin{equation}\label{low}
\liminf_{x \to 0} |x|^{\frac{n-2}{2}} u(x) >0,
\end{equation}
then  $2k<n$, furthermore, there exist $\alpha>0$, $h^{*}\ge h\ge 0$, $\tau$ and $C>0$  such that
\begin{equation} 
u (x) = \left(1 + o(|x|^{\alpha})\right) 
|x|^{-\frac{n-2}{2}} e^{-\frac{n-2}{2}\xi_h(-\ln |x| +\tau)} 
\end{equation}
as  $x \to 0$;

If $2k>n$, or $2k<n$ and 
\begin{equation}\label{rem}
\liminf_{x \to 0} |x|^{\frac{n-2}{2}} u(x) =0,
\end{equation}
then $\lim_{x \to 0} u(x)$ exists and equals some $a>0$,
and there exist some $\alpha>0$ and $C>0$  such that
\begin{equation}\label{hol}
|u (x) - a| \le C |x|^{\alpha};
\end{equation}

If $2k=n$, then there exist some   $0\le h <1$ and  $\alpha>0$ such that
\[
|x|^{\frac{n-2}{2}(1-\sqrt{1-\sqrt[k]{h}})}u(x)
\]
extends to a $C^{\alpha}$ positive function over $B_R$.
\end{corollary}

\begin{remark}
In the case $2k>n$ Gursky and Viaclovsky 
\cite{GV}, YanYan Li \cite{Li06} had obtained \eqref{hol} 
earlier, with
$\alpha=2-\frac nk$. In the case $2k<n$,
M. Gonzalez 
\cite{MG} proved that if $u$ is a solution to \eqref{2} in  $B_R\setminus\{0\}$ in
the $\Gamma_k^+$ class such that $u^{4/(n-2)}|dx|^2$ has finite volume over
$B_R\setminus\{0\}$ , then $u$ is bounded in $B_R\setminus\{0\}$.
\end{remark}

As in \cite{KMPS}, we also obtain 
higher order expansions for solutions to \eqref{1w} in the case $2k\le n$.

\begin{theorem}\label{hodthm}
Let $w(t,\theta)$ be a solution to \eqref{1w} on
$\{t>t_0\}\times \mathbb S^{n-1}$ in the $\Gamma^+_k$ class, where  $n\ge 3$, $2\le k \le n/2$, and the constant $c$
is normalized to be  $2^{-k}\binom{n}{k}$, and let
 $w^*(t)= \xi_h(t+\tau)$ be the radial
solution to \eqref{1w} on $\mathbb R \times \mathbb S^{n-1}$ in the $\Gamma^+_k$ class
for which \eqref{maincon0} holds. Let $\{Y_j(\theta): j=0, 1, \cdots \}$ denote
 the set of normalized spherical harmonics, and $\rho$ be the infimum of the positive characteristic 
exponents defined through Floquet theory
to the linearized equation  of \eqref{1w} at $w^*(t)$ corresponding to higher order spherical 
harmonics $Y_j(\theta)$, $j>n$ --- see the paragraph before Lemma~\ref{la2} in Section 4 for more
detail. Then $\rho >1$, and 
 when $h>0$, there is a 
$$w_1(t,\theta)=\sum_{j=1}^n a_je^{-t-\tau}\left(1+\xi_h^{'} (t+\tau)\right)Y_j(\theta),$$
which is a solution to the linearized equation  of \eqref{1w} at $w^*(t)$, such that
\begin{equation}\label{goal}
|w(t,\theta)-w^*(t) - w_1(t, \theta)|\le C  e^{- \text{min}\{2,\rho\} t} \quad \text{for $t>t_0+1$},
\end{equation}
provided $\rho \ne 2$; when $\rho =2$, \eqref{goal} continues to hold if the right hand side
is modified into $Ct e^{-2t}$.
\end{theorem}

Theorem~\ref{hodthm} requires some knowledge on the  spectrum of the linearized operator to \eqref{1w}.
We are able to provide the needed analysis, and will state them as Propositions~\ref{la1} and~\ref{la4} in section 4.
Such analysis will also be needed in constructing solutions to \eqref{1w} on
$\mathbb S^{n} \setminus \Lambda$, and in analysing the moduli space of solutions to \eqref{1w} on
$\mathbb S^{n} \setminus \Lambda$, when $\Lambda$ is a finite set.
Our knowledge of the spectrum of the linearized operator to \eqref{1w} immediately yields
Fredholm mapping properties of these operators on appropriately defined weighted spaces, as those in
\cite{MPU}, \cite{MS}, and \cite{KMPS}.
We will pursue these problems in a different paper.

It turns out that either of the approaches in \cite{CGS} and \cite{KMPS} can be adapted 
to prove the main part of  Theorom~\ref{main}. 
We will provide proofs along both lines.

The approach in \cite{CGS} first proves that the radial average of the solution is a good approximation
to the solution, and satisfies an ODE which is an approximation to the  ODE \eqref{cylrad} satisfied by a radial
solution to \eqref{cyl}; from this approximate ODE one proves that the radial average is approximated
by a (translated) radial solution to \eqref{cyl}. More specifically, \cite{CGS} first proves \eqref{ave}
for a positive solution to \eqref{sca} in the punctured ball $B_2(0)\setminus\{0\}$.

In terms of $U(t,\theta)=r^{\frac{n-2}{2}}u(r\theta)$, $t=-\ln r = - \ln |x|$, and
\[
\beta (t) := |\mathbb S^{n-1}|^{-1} \int_{\mathbb S^{n-1}} U(t,\theta)d\,\theta,
\]
\eqref{ave} is reformulated as
\begin{equation}\label{ave2}
|U(t,\theta)-\beta (t)| \le C \beta (t) e^{-t}.
\end{equation}

Using gradient estimates and \eqref{ave}, \cite{CGS} further deduces that
for some constant $C>0$
\[
|\grad (u(x) -  \bar u(|x|))| \le C  \bar u(|x|),
\]
which, in terms of $U(t,\theta)$, $t=-\ln r = - \ln |x|$, and $\beta (t)$,
is reformulated as
\begin{equation}\label{grad}
|\grad_{t, \theta} (U(t,\theta) - \beta (t))| \le C \beta (t) e^{-t}.
\end{equation}
It follows from \eqref{cyl}, \eqref{ave2} and a version of  \eqref{grad} for derivatives up to order $2$ that
\begin{equation}\label{7.14'}
\beta^{''}(t)- \frac{(n-2)^2}{4} \beta(t) + \frac{n(n-2)}{4} \beta^{\frac{n+2}{n-2}}(t) 
= O(\beta(t) e^{-t}).
\end{equation}
It is then routine to deduce from \eqref{7.14'} the following approximate first integral, which, up to a 
constant, is referred to
as (7.14) in \cite{CGS}: for some constant $D_{\infty}$,
\begin{equation}\label{7.14}
{\beta^{'}}^2(t)=\frac{(n-2)^2}{4} \left[\beta^2(t)-  \beta^{\frac{2n}{n-2}}(t)\right]+ D_{\infty}+ 
\left( \beta^2(t) + {\beta^{'}}^2(t) \right) O(e^{-t}),
\end{equation}
Since $ \beta(t)$ remains positive for all large $t$, \eqref{7.14} demands that
$0 \ge  D_{\infty} \ge D^*$, where
\[
D^*=-\frac{(n-2)^2}{4} \sup_{\beta \ge 0} 
\left[ \beta^2-  \beta^{\frac{2n}{n-2}}\right]
=- \frac{n-2}{2} \left( \frac{n-2}{n}\right)^{\frac n2}
\]
is determined so that for $0 \ge  D_{\infty} \ge D^*$, 
\[
\sup_{\beta\ge 0} \left(\frac{(n-2)^2}{4}\left[ \beta^2-  \beta^{\frac{2n}{n-2}}\right]+ D_{\infty}\right) \ge 0. 
\]
Then  \cite{CGS} indicates that when $0>  D_{\infty} \ge D^*$, $\beta(t)$ is asymptotic to a translated
solution $\psi (t)$ to \eqref{cylrad} whose first integral is the same as $D_{\infty}$. When $D_{\infty}=0$,
 \cite{CGS} gives a detailed argument  that $0$ is a removable singularity.

We will formulate and prove a general asymptotic approximation
result for solutions to certain ODEs which include the case for scalar curvature and $\sigma_k$ curvature cases.

Consider  a solution $\beta(t)$ to
\begin{equation}\label{pert}
\beta^{''}(t)= f(\beta'(t),\beta(t))+ e_1(t), \quad t\ge 0,
\end{equation}
where $f$ is  locally Lipschitz, and  $e_1(t)$ is considered as a perturbation term with
$e_1(t) \to 0$ as $t\to \infty$ at a sufficiently fast  rate to be specified later.
Suppose that $|\beta(t)|+|\beta'(t)|$ is bounded over $t\in [0,\infty)$. Then
by a compactness argument there exist
a sequence of $t_i \to \infty$ and 
 a solution $\psi(t)$ to
\begin{equation}\label{au}
\psi^{''} (t) = f(\psi'(t),\psi(t))
\end{equation}
which exists for all $t\in \mathbb R$ such that 
\begin{equation}\label{appro}
 \beta(t_i+\cdot)\to \psi
\quad  \text{ in } \ C^1_{loc}(-\infty, \infty)
 \text{  as $i\to  \infty$}.
\end{equation}
\begin{theorem}\label{odethm}
Suppose that,  for the $\beta(t)$,
  $\psi (t)$ and $\{t_i\}$
 above,  $\psi (t)$ is a periodic solution to \eqref{au}
with (minimal) period $T\ge0$. Thus for some finite $m\le M$,
\begin{equation}\label{bd}
\psi (\mathbb R)=[m,M].
\end{equation}
 We may do a time translation for $\psi (t)$ so that $\psi (0)=m,\, \psi'(0)=0$, then
the approximation property \eqref{appro} can be reformulated as,
for some $s$,
\begin{equation} \label{appro2}
\beta(t_i+\cdot))-\psi(-s+\cdot)\to 0,  \text{ in $C^1_{loc}
(-\infty, \infty)$,
as $i\to \infty$.}
\end{equation}
Suppose  that $\psi (t)$ has a first integral in the form of
\begin{equation}\label{H}
H(\psi'(t),\psi(t))=0, \quad \text{for some continuous function $H(x,y)$,}
\end{equation}
where $H$ satisfies the following non-degeneracy condition, depending on
\begin{enumerate}
\item[case (i).] ($\psi (t) \equiv m$ is a constant): there exist  some  $\epsilon_1>0$, $A>0$, $l>0$,
\begin{equation}\label{nond1}
H(x,y)\ge A\left(|x|^l + |y-m|^l\right), \text{ for any $(x,y)$ with $|x|+ |y-m|\le \epsilon_1$;}
\end{equation}
\item[case (ii).]($\psi (t)$ is non-constant): there exist  some  $\epsilon_1>0$, $A>0$ and  $l>0$,
\begin{equation}\label{nond}
|H(0,y)|=|H(0,y)-H(0,m)|\ge A|y-m|^l, \text{ for any $y$ with $|y-m|\le \epsilon_1$.}
\end{equation}
\end{enumerate}
Suppose also that $\beta(t)$  has $H$ as an approximate first integral
\begin{equation}\label{afi}
|H(\beta^{'}(t),\beta(t))| \le e_2(t), \quad \text{for $t\ge 0$,}
\end{equation}
where $e_2(t)\to 0$ as $t \to \infty$.  Without loss of generality, we may suppose that $e_2(t)$ is
monotone non-increasing in $t$. Finally suppose that
\begin{equation}\label{deca}
\int_0^{\infty} \left((e_2(t))^{1/l}+ \sup_{\tau \ge t}|e_1(\tau )| \right) dt < \infty.
\end{equation}
Then, there exist some $s_{\infty}$ and $C>0$ such that,
\begin{equation}\label{conc}
\begin{split}
&|\beta(t)-\psi(t-s_{\infty})|+|\beta^{'}(t)-\psi^{'}(t-s_{\infty})| \\
\le &C \int_{t-1}^{\infty} \left((e_2(t^{'}))^{1/l}+ \sup_{\tau \ge t^{'}}|e_1(\tau )| \right) dt^{'}
\to 0, \quad \text{ as $t \to \infty$.}
\end{split}
\end{equation}

\end{theorem}
In the case of \eqref{sca}, we can take 
\[
H(x,y)=x^2+\frac{(n-2)^2}{4}\left[ y^{\frac{2n}{n-2}}-y^2\right]-D_{\infty},
\]
according to \eqref{7.14}. Then $\beta (t)$ and $H$ satisfy the conditions in Theorem~\ref{odethm}.
We will indicate how Theorem~\ref{odethm} can be applied to prove Theorem~\ref{main}, after we 
provide more background information on solutions to \eqref{1w} and \eqref{1w'}.

Some comments on  Theorem~\ref{odethm}  are appropriate here.
\begin{remark}{}
\begin{enumerate}
\item[(a).] %\eqref{nond} is some kind of non-degeneracy condition on $H$ at $(0,m)$.
 \eqref{nond} would be satisfied with $l=1$ if  $H_y(0,m)\neq 0$, for instance.
Assumptions \eqref{H}, \eqref{nond}, and \eqref{afi} are used only near $(0,m)$, and need not be
posed near the minimum $m$ of $\psi (t)$: the argument would go through if they are posed
near any critical value of $\psi (t)$.
\item[(b).] Our proof  gives an exponential decay rate for 
$|\beta(t)-\psi(t-s_{\infty})|+|\beta^{'}-\psi^{'}(t-s_{\infty})|$ 
when $|e_1(t)|$ and $e_2(t)$ have exponential decay rates.
\end{enumerate}
\end{remark}

Here is a brief description of our plan for the remaining part of the paper.
We will first summarize some needed  preliminary properties for solutions to \eqref{1w} 
 and \eqref{1w'} in section 2,
then provide a proof for Theorem~\ref{main} in section 3, 
using Theorem~\ref{odethm} and several other ingredients, the proof
for which we supply in this section. In section 4, we provide the analysis for the linearized operator
for \eqref{1w} at entire radial solutions, and use them to provide an alternative proof for
Theorem~\ref{main} along the approach of \cite{KMPS}. We will
also provide a proof for Theorem~\ref{hodthm} here. Finally in section 5 we provide a proof for Theorem~\ref{odethm}.

\section{Several  preliminary properties for solutions to \eqref{1}}

To adapt either of the approaches in \cite{CGS} or \cite{KMPS} to our situation, we will need several key properties
of solutions  derived in \cite{Li06}.
The following is a special case of Theorem 1.2 in \cite{Li06}.
\begin{thmD}\label{Li1.2}(\cite{Li06})
 Let $U(t,\theta)$ be any positive solution to \eqref{1} defined on the entire
cylinder $\mathbb R \times \mathbb S^{n-1}$. Suppose that $U^{\frac{4}{n-2}}(t,\theta) (dt^2+d\theta^2)$
is in the $\Gamma_k^+$ class, and
$$
u(x) = |x|^{- \frac{n-2}{2}} U(-\ln |x|, \frac{x}{|x|})
$$
can not be  extended as a $C^2$ positive function near $0$, then
$U$ is independent of $\theta$.
\end{thmD}
\begin{remark} Theorem D is an analogue of Theorem A.
In the setting of Theorem D, it follows from Theorem 1.3 in \cite{LL03} that if $u(x)$
can be extended as a $C^2$ positive function near $0$,
 then $u(x)$ is a 
constant multiple of
$\displaystyle{
( \frac a{ 1+a^2|x-\bar x|^2 })^{ \frac {n-2}2 }
}
$
in $\mathbb R^n$ for some
$\bar x\in \mathbb R^n$ and
$a>0$.
\end{remark}
When a solution $u$ in Theorem D can not be extended as a $C^2$ positive function near $0$, 
we refer to the corresponding $U(t,\theta)=U(t)$ as a global singular positive solution to \eqref{1}.
Using   Theorem C above,
 when $2k<n$, $U(t)$ is a periodic
solution  of \eqref{1}  with $0< U(t) \le 1$ for all
$t \in \mathbb R$ and the first integral $h > 0$.

Another needed estimate, generalizing  estimate \eqref{ave}  from solutions to
\eqref{sca} to solutions to \eqref{1},  is drawn from Theorems 1.1$'$ and 1.3 of \cite{Li06}:
\begin{thmE}(\cite{Li06})
Suppose that $u\in C^2(B_2\setminus\{0\})$ is a positive solution to \eqref{1w'}.
Then
\begin{equation}\label{sup}
\limsup_{x \to 0} |x|^{\frac{n-2}{2}} u(x) < \infty;
\end{equation}
and there exists some constant $C>0$ such that
\begin{equation}\label{ave1}
|u(x) - \bar u(|x|)| \le C|x| \bar u (|x|),
\end{equation}
for $0<|x|\le 1$, where 
\[
 \bar u(|x|) = \frac{1}{|\partial B_{|x|}(0)|} \int_{\partial B_{|x|}(0)} u(y) d\sigma (y)
\]
is the spherical average of $u(x)$ over $\partial B_{|x|}(0)$.
\end{thmE}
As in the previous section, in terms of 
 $t=-\ln r=- \ln |x|$, 
$$U(t,\theta)=r^{\frac{n-2}{2}}u(r\theta)= e^{- \frac{n-2}{2} w(t,\theta)},$$
\[
\beta (t) := |\mathbb S^{n-1}|^{-1} \int_{\mathbb S^{n-1}} U(t,\theta)d\,\theta,
\]
and the spherical average 
\begin{equation}\label{gam}
\gamma (t) := |\mathbb S^{n-1}|^{-1} \int_{\mathbb S^{n-1}}  w(t,\theta) d\,\theta,
\end{equation}
of $ w(t,\theta)$, \eqref{sup} is reformulated as
\begin{equation}\label{sup2}
U(t,\theta) \le C \quad \text{ and } \quad e^{-2w(t,\theta)} \le C.
\end{equation}
We  also derive from \eqref{ave1} that
\begin{equation}\label{ave22}
|U(t,\theta)-\beta (t)| \le C \beta (t) e^{-t},
\end{equation}
and 
\begin{equation}\label{ave3}
|\widehat w(t,\theta)| := |w(t,\theta) - \gamma (t)| \le \tilde C e^{-t}.
\end{equation}
\eqref{ave22} is simply a reformulation of \eqref{ave1} in terms of $U(t,\theta)$ and $\beta (t)$.
In terms of  $w(t,\theta)$, \eqref{ave22} becomes
\[
|e^{-\frac{n-2}{2} w(t,\theta)- \ln  \beta (t)} - 1| \le C e^{-t},
\]
from which it follows that, for some $ \tilde C >0$, 
\begin{equation}\label{est1}
| w(t,\theta)+\frac{2}{n-2} \ln  \beta (t)| \le \tilde C  e^{-t}.
\end{equation}
Integrating over $\theta \in \mathbb S^{n-1}$,
we obtain
\begin{equation}\label{est2}
| \gamma (t) +\frac{2}{n-2} \ln  \beta (t)| \le \tilde C  e^{-t}.
\end{equation}
\eqref{est1} and \eqref{est2} imply \eqref{ave3}.

%\eqref{ave1} is proved in \cite{CGS} for positive singular solution to \eqref{sca}.
We also have a counterpart to \eqref{grad} for positive singular solutions $u(x)$ in 
the $\Gamma^+_k$ class to  \eqref{1w'} on $B_R(0)\setminus \{0\}$.
\begin{proposition}
Let $u(x)$ be a  positive singular solution to  \eqref{1w'} on $B_2(0)\setminus \{0\}$
in the $\Gamma^+_k$ class,  $U(t,\theta)$, $ \beta (t)$, $w(t,\theta)$, and $\gamma (t)$ be defined above.
Then for any $0<\delta$ small, there exists a constant $C>0$ depending on $\delta$ such
that
\begin{equation}\label{grad2}
|\grad_{t, \theta}^{l} (U(t,\theta) - \beta (t))| \le C \beta (t) e^{-(1-\delta) t}, 
\quad \text{for all $t\ge 0$ and $1\le l \le 2$,}
\end{equation}
and
\begin{equation}\label{grad3}
|\grad_{t, \theta}^{l} ( w(t,\theta) - \gamma (t))| \le  C e^{-(1-\delta) t}\quad \text{for all $t\ge 0$ and $1\le l \le 2$.}
\end{equation}
\end{proposition}
We now provide an argument for \eqref{grad3}. First, \eqref{sup2} and the gradient estimates for solutions to  \eqref{1w},
see \cite{GW1}, give a bound $B>0$ depending on $l>0$ and $C$ in \eqref{sup2}, such that
\begin{equation}\label{grad-pre}
|\grad^l_{t,\theta} w(t,\theta)| \le B.
\end{equation}
This obviously leads to 
\[
|\grad^l_{t,\theta} \gamma (t) | \le B,
\]
which, together with \eqref{grad-pre}, implies that
\[
|\grad^l_{t,\theta}  \left( w(t,\theta) - \gamma (t) \right)| \le 2B.
\]
This estimate, together with \eqref{ave3} and interpolation, proves \eqref{grad3}.

\section{First proof of Theorem~\ref{main}: exploiting the ODE satisfied by the radial average}

Let $u(x)$ be a positive
solution to \eqref{1w'} on $B_R \setminus \{0\}$ in the $\Gamma^+_k$ class, where the constant
$c$ is normalized to be  $2^{-k}\binom{n}{k}$, and $\gamma (t)$ is defined as in \eqref{gam}. 
%Suppose that $2k<n$ and $0$ is a non-removable singularity in the sense of \eqref{low}. 
We first make
\begin{claim}
\begin{equation}\label{RE'}
\left\{2(1- \gamma_t^2)^{k-1}\left[ \frac{k}{n} \gamma_{tt} + \frac{n-2k}{2n}(1-\gamma_t^2)\right] 
+ \eta_1(t)\right\}e^{2k\gamma}
=1+ \eta_2(t),
\end{equation}
and
\begin{equation}\label{FI'}
e^{(2k-n) \gamma} \left\{(1-\gamma_t^2)^{k}+\eta_3 (t)\right\}- e^{-n\gamma} \left\{1 +\eta_4 (t)\right\}
= h,
\end{equation}
for some constant $h$, where $\eta_i(t)$, for $i=1,\cdots, 4$, have the decay rate
$\eta_i(t) = O(e^{-2(1-\delta)t})$ as $t\to \infty$, and $\delta>0$ can be made as small as one needs, as in 
\eqref{grad3}.
\end{claim}
  We will postpone a proof for \eqref{RE'} and \eqref{FI'} to the end of this section.
We can think of $h$ as a numerical characteristic to each potential isolated  singularity. 
We make another claim relating  the asymptotic behavior of $u$ with that of $\gamma(t)$, $h$, and $k$.
\begin{claim} Let $u$ be  a positive solution   to \eqref{1w'} in the $ \Gamma_k^+$ class
in a punctured ball $B_R\setminus \{0\}$.
\begin{itemize}
\item[(i)] If \eqref{low} holds, namely
\[
\liminf_{x\to 0} |x|^{\frac{n-2}{2}}u(x) >0,
\]
then $2k<n$ and $h> 0$. Furthermore, for some $\epsilon>0$,
\begin{equation}\label{sgn}
1- \gamma_t^2 (t) \ge \epsilon
\end{equation} for all sufficiently large $t$.
\item[(ii)] In the case $2k<n$,  condition \eqref{rem} holds, namely 
\[
\liminf_{x\to 0} |x|^{\frac{n-2}{2}}u(x) =0
\]
iff $h=0$; in such cases, we furthermore have,  
\begin{equation}\label{nobb}
\lim_{x\to 0} |x|^{\frac{n-2}{2}}u(x) =0, \quad \gamma_t(t) >0 \quad \text{for $t$ large,
 and} \quad \lim_{t\to\infty} \gamma (t) = \infty.
\end{equation}
\end{itemize}
Combining (i) and (ii), we see that in the case $2k<n$, we always have $h\ge 0$, with $h=0$ iff \eqref{rem} holds.
\end{claim}
\begin{proof}
\eqref{sgn}  is proved by noting that \eqref{low} and \eqref{sup2} imply that
\begin{equation}\label{ul}
-C \le \gamma (t) \le C
\end{equation}
for some $C$, which, together with \eqref{RE'}, implies that, for large $t$,
 $1- \gamma_t^2$ never changes sign, which, in turn  with  \eqref{ul}, \eqref{grad-pre} and \eqref{RE'}, implies that, for
some $\epsilon>0$,  $1- \gamma_t^2>\epsilon$ for all sufficiently large $t$. 

The part $h> 0$ in (i)  can be proved in one of two ways. The first proof uses rescaling and compactness arguments
on the translations to $\gamma (t)$, with the help of  \eqref{ul},  \eqref{grad-pre} and \eqref{FI'} to produce
a limiting $\widehat \gamma (t)$ which exists and is bounded for all $t\in \mathbb R$ and satisfies \eqref{3} with 
\[
e^{(2k-n)\widehat \gamma (t)}(1-\widehat \gamma_t^2(t))^k - e^{-n \widehat \gamma (t)} =h
\]
for the same $h$. But the classification result, Theorem C, says that no bounded solution
of \eqref{3} exists for all  $t\in \mathbb R$ with  $h\le 0$. 
Since no  bounded solution exists to \eqref{3} for  all $t\in \mathbb R$
when $2k \ge n$ according  to Theorem C, this argument also shows that \eqref{low} implies that $2k<n$;
equivalently, \eqref{rem} must hold in the case $2k\ge n$.

The second proof regards \eqref{RE'}
as a perturbation of \eqref{3}, and makes a continuous dependence argument,
with the help of \eqref{FI'}, \eqref{sgn},
and \eqref{ul} to prove that, when $h<0$, either $1-\gamma_t^2(t) \to 0$
as $t \to \infty$, which contradicts \eqref{sgn}, or $1-\gamma_t^2(t)\to -\infty$ as
$t \to \infty$ in the case $1-\gamma_t^2(t) <0$ and $k$ is odd, which contradicts \eqref{grad-pre}.
The case $h=0$ can also be ruled out along similar lines by a more careful argument.

\eqref{rem} is equivalent to 
\begin{equation}\label{nob}
\limsup_{t\to\infty} \gamma (t) = \infty.
\end{equation}
So when \eqref{rem} holds and $2k<n$, it follows from  \eqref{nob},
\eqref{FI'} and \eqref{grad-pre} directly that $h=0$. For the converse in (ii), 
when $h=0$, it follows from (i) that \eqref{rem}  must hold, thus proving (ii).
In addition, it follows from \eqref{FI'} that, for sufficiently large $t$, $\gamma_t(t)=0$ 
can occur only near $\gamma(t)=0$.  Together with \eqref{nob}, we see that 
$\gamma_t(t) >0$ for sufficiently large $t$ and $\lim_{t\to \infty} \gamma(t) = \infty$.
\end{proof}

We now proceed to prove  \eqref{maincon}. Our proof will handle four cases
slightly differently: 
\emph{Case (a). $h=0$; Case (b). $h>0$ and $2k <n$;  Case (c). $2k> n$ and $h\ne 0$;
and Case (d).  $2k=n$ and $h\ne 0$}.
Cases (a), (c) and (d) are proved by finding the asymptotics of $\gamma(t)$ directly using \eqref{FI'},
while Case (b) will need the help of Theorem~\ref{odethm}.
\begin{proof}[Case (a). $h=0$]
Using $\lim_{t\to \infty} \gamma(t) = \infty$ 
back into \eqref{FI'}, 
which now takes the form
\[
e^{2k \gamma} \left\{(1-\gamma_t^2)^{k}+\eta_3 (t)\right\}-  \left\{1 +\eta_4 (t)\right\}=0,
\]
we see that $1-\gamma_t^2(t) =: \eta(t) \to 0$
as $t\to \infty$. Since $\gamma_t(t) >0$ for sufficiently large $t$, we conclude that
$1- \gamma_t(t) \to 0$ as $t\to \infty$. As a consequence, $ \gamma(t) \ge (1-\epsilon)t +\gamma_0$
for large $t$ and some $\epsilon>0$ small and $\gamma_0$. This would imply through \eqref{FI'} that
\[
| \eta(t)| \le C e^{ - \frac{2(1-\delta)}{k} t}
\]
for some constant $C>0$ and for large $t$. Finally, we have
\[
|\gamma_t(t)-1| =| \sqrt{1-\eta(t)}-1|\le C e^{ - \frac{2(1-\delta)}{k} t},
\]
from which we conclude that
\[
\gamma(t) - t = \tau+ O(e^{ - \frac{2(1-\delta)}{k} t}),
\]
for some $\tau$ as  $t\to \infty$. 

Note that $\xi_0(t)$, the solution to \eqref{3} with $h=0$, to which  \eqref{RE'} is a perturbation,  
satisfies $\xi_0(t)=t - \ln 2 + O(e^{-2t})$.
Therefore, using also  \eqref{ave3},
\[
w(t,\theta)  = \gamma (t)  + \widehat w(t,\theta)= \xi_0(t+\tau +\ln 2) + O(e^{ - \frac{2(1-\delta)}{k} t}),
\]
as $t\to \infty$, which is \eqref{maincon}. Furthermore
\[
\begin{split}
u(x) &= e^{-\frac{n-2}{2} (w(t,\theta) -t)} 
%\\
 %&
=  e^{-\frac{n-2}{2} \left( \xi_0(t+\tau +\ln 2)-t +  O(e^{ - \frac{2(1-\delta)}{k} t})\right)}\\
&= u^*(|x|) e^{  O(e^{ - \frac{2(1-\delta)}{k} t})}
%\\
%&
=  u^*(|x|) \left( 1+ O(e^{ - \frac{2(1-\delta)}{k} t}) \right)\\
&=  u^*(|x|) \left( 1+O( |x|^{  \frac{2(1-\delta)}{k} })\right)
\end{split}
\]
where 
\[
 u^*(|x|) = e^{-\frac{n-2}{2} \left( \xi_0(t+\tau +\ln 2)-t \right)}
\]
is a positive radial solution to \eqref{1w'} on $\mathbb R^n \setminus \{0\}$.
We also find in this case that 
\[
\lim_{x \to 0} u(x) =  e^{-\frac{n-2}{2} \tau } =: u(0) >0
\]
exists, with
\[
\begin{split}
|u(x)-u(0)| &\le |u(0)| \left|  
e^{-\frac{n-2}{2}\left(\widehat w(t,\theta) 
+ O(e^{ - \frac{2(1-\delta)}{k} t})\right)}-1\right|
%\\
%&
\le C e^{ - \frac{2(1-\delta)}{k} t} \\
&\le C|x|^{\frac{2(1-\delta)}{k}}.
\end{split}
\]
\end{proof}

\begin{proof}[Case (b).  $2k<n$ and $h>0$]
Here $h$ is subject to the  further bound
\[
 h \le h^{*} =\frac{2k}{n-2k}\left( \frac{n-2k}{n} \right)^\frac{n}{2k},
\]
with $h^{*}$  determined by
\[
h^{*} :=\sup\{h: \min_{\gamma} \left(e^{-2k\gamma}+ he^{(n-2k)\gamma}\right) \le 1\}.
\]
Set 
\[
H(x,y)=h+ e^{-n y}-e^{(2k-n)y}(1-x^2)^k.
\]
For $0< h < h^{*}$, $H(0,\xi)=0$ has two simple roots  $\xi_{-}< \xi_{+}$ and $H$ satisfies the
conditions in case (ii) of Theorem~\ref{odethm} with $m= \xi_{-}$ and $l=1$; 
for $h= h^{*}$, $H(0,\xi)=0$ has a double root $m=\xi_{-}=\xi_{+}$,
$H(0, m)=0$ and $H(x,y)\ge 0$ satisfies the conditions in case (i) of
 Theorem~\ref{odethm} as well with $l=2$. Thus, thanks to \eqref{RE'} and \eqref{FI'},  we can apply Theorem~\ref{odethm} to conclude
Theorem~\ref{main} in this case.
\end{proof}

\begin{proof}[Case (c). $2k> n$ and $h \ne 0$] As remarked earlier, \eqref{rem} holds, which implies \eqref{nob}.
\eqref{FI'} implies that, for large $t$,  
$\gamma'(t)=0$ can occur only when $\gamma (t)$ is near  certain finite value.
Together with \eqref{nob}, this implies \eqref{nobb}, which,
together with \eqref{FI'}, implies that 
$(1-\gamma_t^2(t))^k \to 0$ as $t \to \infty$. Then the conclusions \eqref{maincon} and 
\eqref{hol} are  proved in almost identical way as was done above for the
$h=0$ case. 
\end{proof}
\begin{proof}[Case (d). $2k=n$ and $h \neq 0$.]
We first make 
\begin{claim}
If $e^{-2w(t,\theta)}(dt^2+d\theta^2) \in \Gamma_2^+$ for all $\theta \in \mathbb S^{n-1}$ at some $t$, then
\begin{equation}\label{gamma}
1- \gamma_t^2(t) + \mint|\grad \widehat w (t,\theta)|^2 d\theta \ge 0,
\end{equation}
where $\gamma(t)$ and $\widehat w (t,\theta)$ are defined as before.
\end{claim}
 Assuming \eqref{gamma} now, then
\eqref{gamma}, \eqref{grad3} and \eqref{FI'} imply that $0\le h \le 1$. 
The case $h=1$ can be ruled out after further
analysis of  \eqref{FI'}.
We can again argue as above that \eqref{nobb} holds.
%and $\gamma_t (t) >0 $ for $t$ sufficiently large. 
Then \eqref{FI'} implies that 
$(1-\gamma_t^2(t))^k \to h$ as $t \to \infty$; and $e^{-n \gamma(t)} = O(e^{-\alpha t})$ as $t \to
\infty$ for some $\alpha >0$ depending on $0< h <1$.
Now with $\eta (t) := 1- \gamma_t^2(t)$, we find 
$$\eta^k(t) = h+e^{-n\gamma(t)}\left(1+\eta_4(t)\right)-\eta_3(t)
=h +  O(e^{-\alpha t})$$
as $t \to\infty$, and 
\[
\gamma_t(t) = \sqrt{1-\eta (t)} = \sqrt{1- \sqrt[k]{h}} +  O(e^{-\alpha t}),
\]
which implies that $\gamma(t) =  \sqrt{1- \sqrt[k]{h}} t + \gamma_0 +  O(e^{-\alpha t})$, for some
$\gamma_0$. Similarly, $\xi_h(t)$ satisfies
$ \xi_h(t)= \sqrt{1- \sqrt[k]{h}} t + \xi_0 +  O(e^{-\alpha t})$, for some $\xi_0 $. Thus for some
$\tau$, we have
\[
\gamma(t) = \xi_h(t + \tau) +  O(e^{-\alpha t}),
\]
and 
\[
u(x) = e^{-\frac{n-2}{2}\left( w(t, \theta)-t\right)} 
= e^{-\frac{n-2}{2}\left(\gamma (t)-t +\widehat w(t, \theta)\right)},
\]
from which we find that
\[
|x|^{\frac{n-2}{2}\left(1- \sqrt{1- \sqrt[k]{h}}\right)}u(x)
=  e^{-\frac{n-2}{2}\left(\gamma (t)- \sqrt{1- \sqrt[k]{h}} t
+\widehat w(t, \theta)\right)},
\]
extends to a  $C^{\alpha}(B_R)$ positive function  for some $\alpha>0$.
\end{proof}

We now provide proofs for \eqref{RE'}, \eqref{FI'} and \eqref{gamma} in Claims 1 and 3. 
\begin{proof}[Proof of \eqref{RE'}]
\eqref{RE'} is derived from \eqref{3}, \eqref{ave3} and \eqref{grad3}  as follows.
First, with  $\widehat w  (t,\theta) := w(t,\theta) - \gamma (t)$,
it follows from \eqref{ave3} and \eqref{grad3} that
\begin{equation}\label{der1}
\sigma_k (A_{w(t,\theta)})= \sigma_k (A_{\gamma (t)})+ L_{\gamma (t)}[\widehat w (t,\theta)] + \widehat \eta_1  (t,\theta),
\end{equation}
where
$L_{\gamma (t)}$ denotes the linearized operator for $\sigma_k (A_{\gamma (t)})$ at $\gamma (t)$,
and $\widehat \eta_1  (t,\theta)$ satisfies $| \widehat \eta_1  (t,\theta)|=O(e^{-2(1-\delta)t})$.
Next,
\begin{equation}\label{der2}
e^{2kw(t,\theta)}=e^{2k\gamma (t)} \cdot e^{2k \widehat w  (t,\theta)},
\end{equation}
and
\begin{equation}\label{der3}
e^{-2k\widehat w  (t,\theta)} = 1-2k\widehat w  (t,\theta) + \widehat \eta_2  (t,\theta),
\end{equation}
where
\[
| \widehat \eta_2  (t,\theta)| = O(e^{-2t}).
\]
Putting \eqref{der1}, \eqref{der2}, and \eqref{der3} into \eqref{1w}, integrating over $\theta \in \mathbb S^{n-1}$,
and noting that 
\begin{equation}\label{ave0}
\int_{\mathbb S^{n-1}} \widehat w (t,\theta)\, d\theta =0,
\end{equation}
and
\begin{equation}\label{dave0}
\int_{\mathbb S^{n-1}}  L_{\gamma (t)}[\widehat w (t,\theta)]\, d\theta =0,
\end{equation}
we arrive at \eqref{RE'}.
\end{proof}
\begin{proof}[Proof of \eqref{FI'}]
A crude variant of \eqref{FI'} in the case $2k \le n$ can be derived from \eqref{RE'} by elementary means as follows.
Multiplying both sides of \eqref{RE'} by $ne^{-n\gamma (t) }\gamma_t (t)$, one
has
\[
[e^{(2k-n)\gamma (t)}(1-\gamma_t^2(t))^k - e^{-n\gamma (t)  }]_t =  ne^{-n\gamma (t) }\gamma_t (t)\left[e^{2k\gamma (t) }
\eta_1 (t)-\eta_2 (t)\right], 
\]
the right hand of which is of the order $O(e^{-2(1-\delta)t})$ as $t \to \infty$ in the case of $2k \le n$, 
from \eqref{sup2}, the gradient estimates, and the decay rates of $\eta_1 (t)$, $\eta_2(t)$.
It then follows that for some constant $h$, we have
\begin{equation}\label{FI2}
e^{(2k-n)\gamma (t)}(1-\gamma_t^2(t))^k - e^{-n\gamma (t)  } = h + O(e^{-2(1-\delta) t}).
\end{equation}
The more precise version, \eqref{FI'}, is needed only to handle the case $2k>n$, or 
the $h=0$ case when $2k<n$. It is derived from a Pohozaev type identity for solutions $w(t,\theta)$ to 
\eqref{1w} when $\sigma_k$ is a constant, which takes the form
\begin{equation}\label{poh}
\int_{\mathbb S^{n-1}} \left[ \frac{n}{2k\sigma_k} e^{(2k-n)w(t,\theta)} \sum_{a=1}^n 
T_{a1}[w(t,\theta)] w_{at}(t,\theta) - e^{-n w(t,\theta)} \right]\ d\theta = \widehat h
\end{equation} for some constant $\widehat h$ independent of $t$, where $T_{a1}[w(t,\theta)]$ are the components of
the Newton tensor associated with $\sigma_k(A_{w(t,\theta)})$.
Identities of the form \eqref{poh} for solutions to \eqref{1w} were first derived by Viaclovsky in \cite{V00}. 
\eqref{poh} is a version from \cite{H06}.
We assume \eqref{poh} now and postpone a proof to the end of this section. 
Using  \eqref{ave3} and \eqref{grad3}, we find that
\[
\begin{split}
 & \sum_{a=1}^n T_{a1}[w(t,\theta)] w_{at}(t,\theta) \\
= &   T_{11}[\gamma (t) ] \gamma_{tt} + \widehat L_{\gamma(t)}[\widehat w(t,\theta)] + O(e^{-2(1-\delta)t})\\
= & \frac{2k\sigma_k}{n} (1-\gamma_t^2(t))^{k-1}  \gamma_{tt} + \widehat L_{\gamma(t)}[\widehat w(t,\theta)] 
+ O(e^{-2(1-\delta)t}),
\end{split}
\]
where $ \widehat L_{\gamma(t)}$ stands for the linearization at 
$w(t,\theta)=\gamma(t)$ to $ \sum_{a=1}^n T_{a1}[w(t,\theta)] w_{at}(t,\theta)$, and we have used that
$ T_{a1}[\gamma (t) ]=2k\sigma_k \delta_{a1}  (1-\gamma_t^2(t))^{k-1}/n$.
Using  \eqref{RE'} to solve for $(1-\gamma_t^2(t))^{k-1}  \gamma_{tt}$, we find that
\[
(1-\gamma_t^2(t))^{k-1}  \gamma_{tt} = -\frac{n-2k}{2k}(1-\gamma_t^2(t))^{k} + \frac{n}{2k} e^{-2k\gamma(t)}(1+
\eta_2(t))- \eta_1(t).
\]
Using  these in \eqref{poh}, and with estimates like \eqref{der3}, \eqref{ave0} and \eqref{dave0}, we arrive at
\eqref{FI'}, in the case $2k\neq n$, with
\[
\widehat h = \frac{2k-n}{2k} |\mathbb S^{n-1}| h.
\]
When  $2k=n$, \eqref{FI'} is covered by \eqref{FI2}.
\end{proof}

\begin{proof}[Proof of \eqref{gamma}]
\eqref{gamma} is proved by noting that, if $A_g(t,\theta) \in \Gamma^+_2$ for all $\theta \in \mathbb S^{n-1}$, then
\[
\mint A_g(t,\theta)\ d\theta \in  \Gamma^+_2,
\]
due to the convexity of $\Gamma^+_2$.
In our case
the matrix for the Schouten tensor of the metric $g=e^{-2w(t,\theta)}\left(dt^2 + d\theta^2\right)$ is
\begin{equation}\label{matri}
A_g= \begin{bmatrix} 
w_{tt}+w_t^2-\frac 12 |\grad w|^2 - \frac 12 & w_{t\theta_j} + w_t w_{\theta_j} \\
 w_{ \theta_i t} +  w_{\theta_i} w_t & w_{\theta_i \theta_j} + w_{ \theta_i} w_{\theta_j}
+\frac 12(1- |\grad w|^2)\delta_{ij} 
\end{bmatrix}.
\end{equation}
Using $w(t,\theta)=\gamma (t) + \widehat w(t,\theta)$ and $\mint  \widehat w(t,\theta) \ d\theta =0$, we find
\[
\mint A_g(t,\theta)\ d\theta = \text{diag}[\gamma_{tt}(t)-\frac{1-\gamma_t^2(t)}{2}, \frac{1-\gamma_t^2(t)}{2}, \cdots, 
\frac{1-\gamma_t^2(t)}{2} ] - \frac{a(t)}{2}I_{n\times n} +
 \begin{bmatrix} b_{11}(t) & b_{1j}(t) \\ b_{i1}(t) & b_{ij}(t) \end{bmatrix},
\]
where
\[
a(t) = \mint |\grad \widehat w(t,\theta) |^2 \ d\theta,
\]
\begin{gather*}
b_{11}(t)=\mint | \widehat w_t(t,\theta)|^2 \ d\theta, \\
 b_{1i}(t)=  b_{i1}(t) = \mint \widehat w_t(t,\theta)  \widehat w_{\theta_i} (t,\theta) \ d\theta, \quad \text{ for
$i>1$,}\\
 b_{ij}(t) =  \mint \widehat w_{\theta_i}(t,\theta)  \widehat w_{\theta_j} (t,\theta) \ d\theta, \quad \text{ for
$i, j>1$}.
\end{gather*}
But
\[
a(t) I_{n\times n} -  \begin{bmatrix} b_{11}(t) & b_{1j}(t) \\ b_{i1}(t) & b_{ij}(t) \end{bmatrix} \ge 0
\]
as a matrix, so
\[
\text{diag}[\gamma_{tt}-\frac{1-\gamma_t^2(t)}{2}, \frac{1-\gamma_t^2(t)}{2}, \cdots,\frac{1-\gamma_t^2(t)}{2} ] +
 \frac{a(t)}{2}I_{n\times n} \in \Gamma^+_2.
\]
Computing the $\sigma_1$ and $\sigma_2$ of this tensor as in \eqref{2} it follows  that
\[
\gamma_{tt}(t) + \frac{n-2}{2}\left(1-\gamma_t^2(t)-a(t)\right) > 0,
\]
and
\[
\left(1-\gamma_t^2(t)+a(t)\right) \left[ 1-\gamma_t^2(t)+a(t)+\frac{4}{n} \left( \gamma_{tt}(t) + \gamma_t^2(t) -1 \right)
\right] >0.
\]
Simple algebra from these two inequalities implies \eqref{gamma}.
\end{proof}
Finally we sketch a proof for \eqref{poh}.  It follows from equation (3) in \cite{H06} that $\grad_a Y^a =0$,
where
\[
Y^a= T^a_b \grad^b \left(\text{div}_g X\right) + 2k\sigma_k X^a,
\]
for any conformal Killing vector field $X^a$ on $(M,g)$ with $\sigma_k(A_g)\equiv$ constant on $M$.
We will take $M=\mathbb R \times \mathbb S^{n-1}$,
$g=e^{-2w(t,\theta)}(dt^2+d\theta^2)$, and $X=\partial_t$. Thus
\begin{equation} \label{pohd}
\int_{\mathbb S^{n-1}} Y^1(t,\theta) \sqrt{g(t,\theta)} d\theta = \text{constant},
\end{equation}
independent of $t$.  In addition,
$\text{div}_g X= - n w_t(t,\theta)$,
and \eqref{pohd} would take the form
\[
\int_{\mathbb S^{n-1}} \left( -n \sum_{b=1}^n T_{1b} w_{tb}(t,\theta) e^{2k w(t,\theta)} +  2k\sigma_k \right)
e^{-n  w(t,\theta)} d\theta = \text{constant},
\]
which gives \eqref{poh}.

\section{Proof for Theorem~\ref{hodthm} and second proof for Theorem~\ref{main}}

We can now present our second proof for the main part of Theorem~\ref{main}: the case $2k<n$ and $h>0$.
Let $w(t,\theta)$ be a  solution to $\sigma_k(g^{-1}\circ A_g) = 2^{-k} \binom{n}{k}$, 
with 
$$g=u^{\frac{4}{n-2}}(x)|dx|^2=e^{-2w(t,\theta)}(dt^2 +d\theta^2)$$
 for $x$ to be over the
punctured unit ball $x \in B^n\setminus \{ 0\}$. It is assumed that $g$ is in the
$\Gamma_k^+$ class over $ B^n\setminus \{ 0\}$. Then $w(t,\theta)$ is defined for 
$(t,\theta) \in \mathbb R^+\times \mathbb S^{n-1}$, as $t= -\ln |x|$. It follows from Theorem E
that \eqref{sup2} holds, i.e., for some constant $C_2>0$, 
%\begin{equation}\label{upp}
\[
e^{-2w(t,\theta)} \le C_2
\]
for all $(t,\theta) \in \mathbb R^+\times \mathbb S^{n-1}$. 
It follows from our discussion in the beginning of the previous section that $h>0$ implies  \eqref{low}, i.e.,  for some $C_1>0$
%\begin{equation}\label{lowb}
\[
e^{-2w(t,\theta)} \ge C_1
\]
for $(t,\theta) \in \mathbb R^+\times \mathbb S^{n-1}$, namely,
$w(t,\theta)$ is bounded over $(t,\theta) \in \mathbb R^+\times \mathbb S^{n-1}$. As in \cite{KMPS}, We  make the
following assertions about the behavior of $w(t,\theta)$ as $t \to \infty$.
\begin{enumerate}
\item[(a)] Let $t_j\to \infty$ be any sequence tending to $\infty$, then $\{w_j(t,\theta) := w(t+t_j, \theta)\}$
has a subsequence converging to a 
bounded limiting solution $\xi(t)$ of \eqref{3} defined for
$(t,\theta) \in \mathbb R\times \mathbb S^{n-1}$. The convergence 
is uniform on any compact subset of $\mathbb R^+\times \mathbb S^{n-1}$.
\item[(b)] Any angular derivative $\partial_{\theta} w(t,\theta)$ of $w$ converges to $0$ as $t \to \infty$.
\item[(c)] There exists $S>0$ such that for any infinitesimal rotation $\partial_{\theta}$ of $\mathbb S^{n-1}$,  and for
any $t_j\to \infty$, if we set $A_j = \sup_{t\ge 0} |\partial_{\theta} w_j (t,\theta)|$, and 
 if $|\partial_{\theta} w_j (s_j, \theta_j)|= A_j$ for some
$(s_j, \theta_j) \in \mathbb R^+\times \mathbb S^{n-1}$, then $s_j \le S$.
\item[(d)] $\partial_{\theta} w(t,\theta)$ converges to $0$ at an exponential rate as $t \to \infty$, and
$$|w(t,\theta)-|\mathbb S^{n-1}|^{-1}\int_{\mathbb S^{n-1}} w(t,\omega)d\omega|$$ 
converges to $0$ at an exponential rate as $t \to \infty$.
\item[(e)] There exists a bounded (periodic) solution $\xi (t)$ of \eqref{3} and $\tau \ge 0$ such that
$w(t,\theta)$ converges to $\xi (t +\tau )$  at an exponential rate as $t\to \infty$.
\end{enumerate}

(a)--(e) are proved along almost identical lines as in the proof for 
Proposition 5 in \cite{KMPS}, provided some analytical preparations are established.
In our case here, proofs for (a) and (b) can be provided 
using the local derivative
estimates of \cite{GW1} for solutions of \eqref{1} and the classification 
result in \cite{Li06}, reformulated as Theorem D above; proofs for (c) , (d) and (e) will need
an analysis of the linearized operator of \eqref{1w} at a radial solution $\xi (t)$, as characterized by
Proposition~\ref{la1} below. 
We remark that an alternative proof for (d) without using the analysis of
the linearized operator is in fact already contained in  \eqref{ave3} and \eqref{grad3}.

%The matrix for the Schouten tensor of the standard cylindrical metric  
%$dt^2 + d\theta^2$ on  $\mathbb R \times \mathbb S^{n-1}$ is 
%\[
%\text{diag} \begin{bmatrix}
%-\frac 12 & \frac 12 & \cdots & \frac 12 \end{bmatrix},
%\]
To compute the linearized operator of \eqref{1w} at a radial solution $\xi (t)$, we use \eqref{matri}.
When $w(t,\theta)=\xi(t)$, $A_g$ becomes a block diagonal matrix
\[
\begin{bmatrix}
\xi_{tt}+ \frac 12 (\xi_t^2-1) & 0 \\
0& \frac 12(1- | \xi_t|^2)\delta_{ij}
\end{bmatrix}.
\]
When we linearize $\sigma_k(A_g)$ at such a block diagonal matrix, the coefficient matrix consisting of
the coefficients of the Newton tensor
\[
T_{ij} = \frac{1}{(k-1)!} \delta^{i_1 \ldots i_{k-1} i}_{j_1 \ldots j_{k-1} j} A_{i_1j_1} \ldots
A_{ i_{k-1}  j_{k-1}}
\]
is also diagonal:
\[
T_{11}= \binom{n-1}{k-1} \frac {1}{2^{k-1}}(1- |\xi_t|^2)^{k-1},
\]
while for $i\ge 2$, 
\[
\begin{split}
T_{ii} &=  \binom{n-1}{k-1} \frac {(n-k)(1- |\xi_t|^2)^{k-1}}{2^{k-1}(n-1)}
+  \binom{n-2}{k-2} \frac {(1- |\xi_t|^2)^{k-2}}{2^{k-2}}\left[\xi_{tt}+ \frac 12 (\xi_t^2-1)\right]
\\
&=   \binom{n-1}{k-1}  \frac {(1- |\xi_t|^2)^{k-2}}{2^{k-2}(n-1)}\left[
(k-1) \xi_{tt} + \frac{n-2k+1}{2} \left(1- |\xi_t|^2\right)\right].
\end{split}
\]
So the linearization of $\sigma_k (A_g)$ at $g=e^{-2\xi(t)}(dt^2 + d\theta^2)$ is
\[
\begin{split}
L_{\xi}(\phi) &= T_{11}(t)\left[\phi_{tt}(t,\theta)+\xi_t (t) \phi_t(t,\theta)\right] + \sum_{i\ge 2} T_{ii}(t) 
\left[\phi_{\theta_i\theta_i}(t,\theta) - \xi_t (t) \phi_t(t,\theta) \right] \\
 &= T_{11}(t) \phi_{tt} (t,\theta)+ \left[  T_{11}(t) -(n-1) T_{22}(t)\right] \xi_t (t) \phi_t (t,\theta)+ 
 T_{22}(t) \Delta_{\theta} \phi(t,\theta), \\
&= \frac{(1- |\xi_t (t)|^2)^{k-2}}{2^{k-2}}\binom{n-1}{k-1} \left[ A(t) \phi_{tt}(t,\theta)+
B(t) \phi_t(t,\theta)+C(t)  \Delta_{\theta} \phi (t,\theta)\right],
\end{split}
\]
where
\begin{gather}
A(t)= \frac{(1- |\xi_t (t)|^2)}{2}, \\
B(t) = - \xi_t (t) \left[ (k-1) \xi_{tt} (t) + \frac{n-2k}{2} (1- |\xi_t (t)|^2) \right] ,\\
C(t) = \frac{k-1}{n-1} \xi_{tt}(t) + \frac{n-2k+1}{n-1} \cdot \frac{1- |\xi_t(t)|^2}{2}.
\end{gather}

When $\xi(t)$ is a solution to $\sigma_k(g^{-1}\circ A_g) = \text{const.}$, normalized to be $2^{-k} \binom{n}{k}$,
the linearization of the nonlinear PDE \eqref{1w} at $\xi(t)$ is then
\[
L_{\xi}(\phi) + 2^{1-k} k \binom{n}{k} e^{-2k \xi(t)} \phi (t,\theta) =0.
\]
If we take the projections of $\phi (t, \cdot)$ into spherical harmonics:
\[
\phi (t, \theta) = \sum_j \phi_j(t) Y_j(\theta), \quad \text{where $Y_j(\theta)$ are the normalized eigenfunctions
of $\Delta_{\theta}$ on $L^2(\mathbb S^{n-1})$.}
\]
then
$\phi_j(t)$ satisfies the ODE
\begin{equation} \label{*}
L_j [\phi_j] :=
  \phi_j^{''}(t)+ \frac{B(t)}{A(t)} \phi_j^{'}(t) +\left\{-\lambda_j \frac{C(t)}{A(t)}+\frac{ n e^{-2k\xi(t)}}
{2A(t)(1-\xi_t^2(t))^{k-2}}\right\}\phi_j(t)=0,
\end{equation}
where $\lambda_j$ are the eigenvalues of $\Delta_{\theta}$ on $L^2(\mathbb S^{n-1})$ associated with
$Y_j(\theta)$, thus 
\[
\lambda_0=0,\quad \lambda_1=\cdots=\lambda_n=n-1, \quad \lambda_j\ge 2n, \quad\text{for $j > n$.}
\]
%where the coefficient of $\phi_j(t)$ is
%\[
%-\lambda_j\left[\frac{k-1}{n-1} \xi_{tt}(t) + \frac{n-2k+1}{n-1} \cdot \frac{1- |\xi_t(t)|^2}{2}\right]
%+ \frac{ n e^{-2k\xi(t)}}
%{2(1-\xi_t^2(t))^{k-2}}.
%\]
Similar to properties of the linearized operator to the scalar curvature operator used in \cite{KMPS},
we have the following properties for the $L_j$'s.
\begin{proposition} \label{la1}
For all solutions $\xi_h (t)$ to \eqref{3} with $h>0$, $k<n$, and $j\ge 1$, the following holds:
	\begin{enumerate}
	\item[(i)] $L_j [\phi]=0$ has a pair of linearly
independent solution basis on $\mathbb R$, one of which grows unbounded and the other one
decays exponentially as $t\to \infty$;
	\item[(ii)] Any solution of $L_j [\phi]=0$ which is bounded for  $\mathbb R^+$ must decay exponentially;
	\item[(iii)] Any solution of $L_j [\phi]=0$ which is bounded for all of $\mathbb R $ must be identically $0$;
	\item[(iv)] Any solution of $L_j [\phi]=0$ which is bounded for all of $\mathbb R^+$ must be 
unbounded on $\mathbb R^-$.
	\end{enumerate}
These conclusions remain true for solutions $\xi_h (t)$ to \eqref{3} with $h=0$ and $\lambda_j \ge 2n$.
\end{proposition}
While Proposition~\ref{la1} is sufficient for providing a proof for Theorem~\ref{main}, Theorem~\ref{hodthm} 
requires some more detailed knowledge about the linearized operator to \eqref{1w}. More
specifically, the decay rates of bounded solutions to $L_j [\phi]=0$ on $\mathbb R^+$ need to be faster than
$e^{-t}$ when $\lambda_j \ge 2n$. $L_j$ is an ordinary differential operator with period coefficient,
so,  by Floquet theory,
 has a set of  well defined characteristic roots which give the exponential decay/grow rates
to solutions $\phi$ of $L_j [\phi]=0$ on $\mathbb R$, see, for instance, Theorem 5.1 in Chapter 3
of \cite{CL}. In fact, Theorem 5.1 in Chapter 3 of \cite{CL} and (5.11) on p. 81 of \cite{CL}
applied to $L_j$ implies that $L_j [\phi]=0$ has a set of fundamental solutions in the 
form of $e^{\rho_j t} p_1(t)$ and $e^{-\rho_j t}p_2(t)$ for some periodic functions $p_1(t)$ and $p_2(t)$,
when $\rho_j \ne 0$.  We have the following
\begin{lemma}\label{la2}
When $2k \le n$ and $h>0$, there is a $\beta_* > \sqrt{2}$ such that for all $\lambda_j \ge 2n$, the associated $\rho_j$
satisfies $\rho_j \ge \beta_*$.
\end{lemma}
We can also formulate and prove a version that does not need $L_j$ to have the structure to apply the
Floquet theory.
\begin{lemma}\label{la3}
Define
\[
V(t)= e^{(1-\frac{n}{2k})\xi(t)}\left( e^{-n\xi(t)} +h\right)^{\frac{k-1}{2k}}.
\]
Then
\begin{equation}\label{odetrans}
V(t)  L_j[V^{-1}(t) \psi(t)] = \psi_{tt}(t)+ E(t)  \psi(t),
\end{equation}
where we can estimate $E(t)\le - C_n < -2$ when $\lambda_j \ge 2n$ and $2k\le n$.
As a consequence, when $\beta^2< C_n$, $e^{-\beta t} V^{-1}(t)$ satisfies
\[
 L_j[ e^{-\beta t} V^{-1}(t)] =\left[ \beta^2 + E(t) \right]  e^{-\beta t} V^{-1}(t)
\]
is a supersolution to $L_j[\phi]=0$ on $\mathbb R^+$, therefore
when $2k \le n$ and $h>0$, for any $\beta <  \beta_* :=\sqrt{C_n}$, and for all $\lambda_j \ge 2n$,  
any bounded solution $\phi$ of  $L_j [\phi]=0$ on $\mathbb R^+$ satisfies $|\phi(t)|\lesssim e^{-\beta t}$.
\end{lemma}
\begin{remark} 
Lemma~\ref{la2} is an immediate consequence of \eqref{odetrans} in Lemma~\ref{la3}.
It is not immediately clear that the  characteristic roots  $\rho_j$ of $L_j$ are monotone
increasing as the $\lambda_j$ increases.  But in the case
$2k \le n$ and $h>0$, \eqref{odetrans} allows a variational construction
of a bounded (in fact, decaying) fundamental solution to $L_j[\phi]=0$ on $\mathbb R^{+}$ which 
is positive.  Thus in such cases comparison theorems show that the $\rho_j$ indeed is monotone
increasing as $\lambda_j$ increases.
\end{remark}
With such knowledge, we can now establish
\begin{proposition}\label{la4}
 Suppose that $\phi (t,\theta)$ satisfies
\begin{equation}\label{lindeca}
L_{\xi}(\phi) + 2^{1-k}k \binom{n}{k} e^{-2k\xi(t)}\phi(t,\theta)=
r(t,\theta),\quad \text{for $t\ge t_0$ and $\theta \in \mathbb S^{n-1}$.}
\end{equation}
Suppose that for some $0<\beta< \beta_*$ and $\beta \ne 1$, $|r(t,\theta)| \lesssim e^{-\beta t}$.
Then there exist constants $a_j$ for $j=1,\cdots, n$, such that
\begin{equation}\label{app}
|\phi (t,\theta) -\sum_{j=1}^n a_j e^{-t}(1+\xi_t(t))Y_j(\theta)|\lesssim e^{-\beta t}.
\end{equation}
In fact, when $\beta_* \le \beta <\rho_{n+1}$, \eqref{app} continues to hold, and when $\beta  > \rho_{n+1}$,
we will have
\begin{equation}\label{app2}
|\phi (t,\theta) -\sum_{j=1}^n a_j e^{-t}(1+\xi_t(t))Y_j(\theta)|\lesssim e^{-\rho_{n+1} t}.
\end{equation}
When $\beta =1$, \eqref{app} continues to hold if the right hand side is modified into
$t e^{-t}$; and when $\beta = \rho_{n+1}$, \eqref{app2} continues to hold if the right hand side is modified into
$t e^{-\rho_{n+1} t}$.
\end{proposition}
We now first prove Proposition~\ref{la1}, then provide a proof for Proposition~\ref{la4} and Theorem~\ref{hodthm},
while deferring the proof for Lemma~\ref{la3} to the appendix.
For $j\le n$, (i)---(iv) of Proposition~\ref{la1} follow from an  explicit solution basis to \eqref{*}; 
for $j\ge n+1$, the arguments in \cite{KMPS} 
relies on the
sign of the coefficient of the zeroth order term of $L_j$ to be negative. Our computations below  verify
the same properties for the $\sigma_k$ curvature problem when $k<n$, although we will only use
these properties for  the case $2k<n$ here.

Since $\xi(t)$ satisfies \eqref{3}, with $\sigma_k=2^{-k}\binom{n}{k}$, \eqref{3} becomes
\begin{equation}\label{krad}
2(1-\xi_t^2)^{k-1} \left[\frac{k}{n} \xi_{tt} + (\frac{1}{2} -\frac{k}{n})(1-\xi_t^2)\right]e^{2k\xi}
=1.
\end{equation}
Due to the translation invariance of \eqref{krad} in $t$, $\phi_0^+=\partial_t \xi(t)$ is a solution to
\eqref{*} for $\lambda_0=0$; since $\xi_{h}(t)$ is another family of solution to
\eqref{krad}, we find $\phi_0^-=\partial_h \xi_{h}(t)$ to be another solution to \eqref{*} for $\lambda_0=0$.
Differentiating the first integral $ e^{(2k-n)\xi}(1-\xi_t^2)^k=e^{-n\xi} +h$ with respect to $t$ and $h$, 
respectively, one finds that 
$\{ \phi_0^+(t), \phi_0^-(t)\}$ is linearly independent, thus
forms a solution basis to \eqref{*} for $\lambda_0=0$.

\begin{proof}[Proof of Proposition~\ref{la1}]
For $h>0$, and $\lambda_j=n-1$, which corresponds to $Y_j(\theta)=\theta_j$, we claim that
\[
\left[1-\partial_t \xi (t)\right] e^t  \quad \text{and} \quad 
\left[1+\partial_t \xi (t)\right] e^{-t} 
\]
form a basis to \eqref{*}. This is due the translation invariance of \eqref{1}: if $u(x)$ is a solution
to \eqref{1}, so is $u(x+a)$ for any $a \in \mathbb R^n$. In terms of $w(t,\theta)$, this means that
\[
w_a(t,\theta) := -\ln |x| -\frac{2}{n-2} \ln u (x+a)
\]
is a solution to \eqref{1w}. Thus $\partial_{a_j}|_{a=0} w_a(t,\theta) $ is a solution to  the linearized
equation of \eqref{1w}. But when $w(t,\theta)=\xi (t)$, we have
\[
\begin{split}
\partial_{a_j}|_{a=0} w_a(t,\theta)
%\\
&=
%&
 -\frac{2}{n-2} \partial_{x_j} \ln u (x) 
=\partial_{x_j} \left[ \ln |x| + w(t, \theta)\right] \\
&=
%&
\frac{x_j}{|x|^2} + \partial_t \xi (t) \frac{\partial t}{\partial x_j} 
= \left[1-\partial_t \xi (t)\right] e^t \theta_j.
\end{split}
\]
Thus $\left[1-\partial_t \xi (t)\right] e^t $ is a solution of 
\eqref{*} with $\lambda_j=n-1$.  Since we have normalized $\xi(t)$ such that it is even in $t$, 
we find that $\left[1+\partial_t \xi (t)\right] e^{-t} $ is another solution of 
\eqref{*} with $\lambda_j=n-1$. $\{ \left[1-\partial_t \xi (t)\right] e^t, \left[1+\partial_t \xi (t)\right] e^{-t}\}$ 
becomes linearly dependent only when they are identical (as they are both equal to $1$ at $t=0$), 
which is the case only when $\partial_t \xi (t) = \tanh (t)$ and $h=0$. 
When $h>0$ and $k<n$, one can use the asymptotic expansion in \cite{CHY05} to see that
$\left[1-\partial_t \xi (t)\right] e^t$ is exponentially growing. When
 $h>0$ and $2k \le n$, this can be seen more directly: the  solution $\xi(t)$ has the bound $C_h^{-1}\le
1\pm \partial_t \xi (t) \le C_h$ for some $C_h>0$,
so $\{ \left[1+\partial_t \xi (t)\right] e^{-t}, \left[1-\partial_t \xi (t)\right] e^t \}$ forms
a  solution basis for \eqref{*} with $\lambda_j=n-1$, with one exponentially decaying and the other one
 exponentially growing, and
the conclusion of the Proposition in the case $\lambda_j=n-1$ follows from the explicit basis.

For $\lambda_j \ge 2n$, we will verify that 
\begin{equation}\label{nup}
\text{the coefficient of $\phi_j$ in \eqref{*} has a negative
upper bound.}
\end{equation}
Assuming \eqref{nup} for now, we sketch 
a proof for properties (i)-(iv) of $L_j$ for the case $\lambda_j \ge 2n$.
The key is to check that for $0<\lambda$ small, 
$e^{\pm \lambda t}$ are supersolutions of $L_j[e^{\pm \lambda t}]\le0$.
This is because
\[
L_j[e^{\pm \lambda t}]
= \lambda^2 \pm\lambda \left[1-(n-1)\frac{C(t)}{A(t)}\right]\xi_t(t) -\lambda_j \frac{C(t)}{A(t)}+
n e^{-2k\xi(t)}(1-\xi_t^2)^{1-k},
\]
and it follows from \eqref{nup}  that for $\lambda_j \ge 2n$,
\[
-\lambda_j \frac{C(t)}{A(t)}+n e^{-2k\xi(t)}(1-\xi_t^2)^{1-k} \\
\]
has a negative upper bound. Furthermore, 
 using \eqref{krad}, 
we have
\[
\begin{split}
C(t)=&\frac{k-1}{n-1} \xi_{tt}(t) + \frac{n-2k+1}{n-1} \cdot \frac{1- |\xi_t(t)|^2}{2}\\
=& \frac{n(k-1)}{k(n-1)} \left[\frac{ e^{-2k\xi}}{2(1-\xi_t^2)^{k-1}} - \frac{n-2k}{2n} (1-\xi_t^2)\right]
+ \frac{n-2k+1}{2(n-1)} (1-\xi_t^2) \\
=& \frac{n(k-1)}{2k(n-1)} \frac{ e^{-2k\xi}}{(1-\xi_t^2)^{k-1}} +\frac{(n-k)(1-\xi_t^2)}{2k(n-1)}.
\end{split}
\]
Thus
\[
\begin{split}
\frac{C(t)}{A(t)} &= \frac{n(k-1)}{k(n-1)} \frac{ e^{-2k\xi}}{(1-\xi_t^2)^{k}} +\frac{(n-k)}{k(n-1)}\\
&= \frac{n(k-1)}{k(n-1)} \frac{ e^{-n\xi}}{ e^{-n\xi} +h} +\frac{(n-k)}{k(n-1)}
\end{split}
\]
is bounded from above. Here we used the first integral $ e^{(2k-n)\xi}(1-\xi_t^2)^k=e^{-n\xi} +h$
and as a consequence $\xi \ge 0$.

% When $2k<n$, \eqref{FI} also implies that $he^{(n-2k)\xi} \le 1$, so $h e^{n\xi}$ has an upper bound
%which may depend on $h>0$. 
%\[
%\frac{n}{2k}\left[ \frac{n}{2k}-1\right]^{\frac{2k}{n}-1} h^{\frac{2k}{n}} \le (1-\xi_t^2)^k \le 1,
%\]
It is now clear that  we can choose $\lambda >0$ small to make $L_j[e^{\pm\lambda t}]<0$ for all $t \in \mathbb R$.

Now fix such a $\lambda >0$.
We claim that if $\phi (t)$ is a bounded solution of $L_j[\phi]=0$ on $\mathbb R^{\pm}$, then
\[
|\phi (t) | \le |\phi (0)| e^{-\lambda (\pm t)}\quad\text{for all $t \in \mathbb R^{\pm}$},
\]
which then implies (ii).
This is because for any $\epsilon >0$, $|\phi (0)| e^{-\lambda (\pm t)} + \epsilon e^{\lambda (\pm t)}$ is a supersolution
of $L_j$ on $\mathbb R^{\pm}$. So if  $\phi (t)$ is a bounded solution of $L_j[\phi]=0$ on $\mathbb R^{\pm}$,
then by comparison principle,
\[
|\phi (t) | \le |\phi (0)| e^{-\lambda  (\pm t)} + \epsilon e^{\lambda (\pm t)}\quad\text{for all $t \in \mathbb R^{\pm}$}.
\]
For any fixed $t \in \mathbb R^{\pm}$, since the above estimate holds for all $\epsilon >0$, we can send
$\epsilon$ to $0$ to verify our claim. (iii) now follows from (ii) and the maximum principle, and (iv) obviously
is a direct corollary of (iii).

Next,  any $L_j$ has a pair of linearly independent solution basis $\{\phi_1(t), \phi_2(t)\}$ on $\mathbb R$. If both
are bounded on  $\mathbb R^+$, and $a$ and $b$ are such that $a \phi_1(0)+ b  \phi_2(0)=0$, then
our claim implies that $a\phi_1(t) + b \phi_2(t) \equiv 0$ on  $\mathbb R^+$, contradicting their choice.

It remains to establish that there is a nontrivial solution of $L_j[\phi]=0$ bounded on $\mathbb R^+$.
Since we have verified  that $L_j$ is uniformly elliptic on $\mathbb R^+$
and satisfies the maximum principle there, we can establish the desired existence by a convergence 
argument for solutions which are constructed on a sequence of finite intervals that exhaust $\mathbb R^+$.

Finally we come back to verify \eqref{nup}.
Since $1\ge 1-\xi_t^2 >0$ by Theorem C (this is also true for $k=1$), we see that, when  $\lambda_j \ge 2n$,
 the coefficient of $\phi_j(t)$ in \eqref{*} is bounded from above by
\[
\begin{split}
&- \lambda_j\left[\frac{n(k-1)}{k(n-1)} \frac{ e^{-2k\xi}}{(1-\xi_t^2)^{k}} +\frac{(n-k)}{k(n-1)}\right]
+ \frac{ n e^{-2k\xi(t)}} {(1-\xi_t^2(t))^{k-1}}\\
\le & -2n \left[\frac{n(k-1)}{k(n-1)} \frac{ e^{-2k\xi}}{(1-\xi_t^2)^{k}} +\frac{(n-k)}{k(n-1)}\right]
+ \frac{ n e^{-2k\xi(t)}} {(1-\xi_t^2(t))^{k-1}}\\
=&- \frac{2 n e^{-2k\xi(t)}}{(n-1) (1-\xi_t^2)^{k}} \left[ \frac{n(k-1)}{k}- \frac{(n-1)(1-\xi_t^2)}{2}\right]
-\frac{2n(n-k)}{k(n-1)}\\
\le &-  \frac{2n e^{-2k\xi(t)}}{(n-1) (1-\xi_t^2)^{k}} \left[ \frac{n}{2}-\frac{n}{k} + \frac 12 
\right] -\frac{2n(n-k)}{k(n-1)} \\
< &- \frac{2n(n-k)}{k(n-1)}< 0,
\end{split}
\]
when $n> k\ge 2$; when $k=1$, the above estimate gives
\[
\begin{split}
&- \lambda_j\left[\frac{n(k-1)}{k(n-1)} \frac{ e^{-2k\xi}}{(1-\xi_t^2)^{k}} +\frac{(n-k)}{k(n-1)}\right]
+ \frac{ n e^{-2k\xi(t)}} {(1-\xi_t^2(t))^{k-1}}\\
\le &- \frac{2 n e^{-2k\xi(t)}}{(n-1) (1-\xi_t^2)^{k}} \left[ \frac{n(k-1)}{k}- \frac{(n-1)(1-\xi_t^2)}{2}\right]
-\frac{2n(n-k)}{k(n-1)}\\
= & n\left[ e^{-2\xi (t)} -2\right] 
%\\
\le -n,
\end{split}
\]
as $\xi (t) \ge 0$, which follows from the first integral
\[
e^{-2\xi (t)} + h e^{(n-2)\xi (t)} = 1-\xi_t^2(t) \le 1
\]
with $h\ge 0$; while for $k=n$ and $h=0$,  the above estimate gives
\[
\begin{split}
&- \lambda_j\left[\frac{n(k-1)}{k(n-1)} \frac{ e^{-2k\xi}}{(1-\xi_t^2)^{k}} +\frac{(n-k)}{k(n-1)}\right]
+ \frac{ n e^{-2k\xi(t)}} {(1-\xi_t^2(t))^{k-1}}\\
\le &-  \frac{2n e^{-2n\xi(t)}}{ (1-\xi_t^2)^{n}} \left[1 -\frac{1-\xi_t^2}{2} \right] 
%\\
\le 
%&
- \frac{n e^{-2n\xi(t)}}{(1-\xi_t^2(t))^{n}}\\
= &- n,
\end{split}
\]
from the first integral
\[
e^{n \xi(t)} (1-\xi_t^2(t))^{n} - e^{-n \xi(t)} =h=0.
\]

\end{proof}

Next is a proof for Proposition~\ref{la4}.
\begin{proof}[Proof for Proposition~\ref{la4}] 
Define
\[
\wph = \phi (t,\theta) - \sum_{j=0}^n \pi_j[\phi (t,\theta)]Y_j(\theta),
\]
where $\phi_j(t) :=  \pi_j[\phi (t,\theta)]$ is the $L^2$ orthogonal projection of 
$\phi (t,\theta)$ onto $\text{span}\{Y_j(\theta)\}$.
Then
\begin{equation}\label{ortho1}
\int_{\mathbb S^{n-1}} \wph Y_j(\theta)\, d\theta= \int_{\mathbb S^{n-1}} \grad \wph \cdot \grad Y_j(\theta)\, d\theta
= \int_{\mathbb S^{n-1}} \Delta_{\theta} Y_j(\theta)\wph \, d\theta=0,
\end{equation}
for $j=0, \cdots, n$.  As a consequence,
\begin{equation}\label{ortho2}
\left\{
\begin{aligned}
\int_{\mathbb S^{n-1}} \Delta_{\theta} \phi (t,\theta) \wph \, d\theta&= 
- \int_{\mathbb S^{n-1}}|\grad_{\theta} \wph|^2\, d\theta, \\
\int_{\mathbb S^{n-1}}  \phi_t (t,\theta) \wph \, d\theta&= \int_{\mathbb S^{n-1}} \widehat \phi_t(t,\theta) \wph \, d\theta
= \frac 12 \frac{d}{dt} \int_{\mathbb S^{n-1}} |\wph|^2 \, d\theta.\\
\end{aligned}
\right.
\end{equation}
In the following
we will prove separately the expected decays for $\widehat \phi (t,\theta)$ and 
$\phi_j(t) := \pi_j[\phi (t,\theta)]$, for $j=0,1,\cdots, n$.
We first  estimate $\phi_j(t)=\pi_j[\phi(t,\theta)]$ for $j=0,\cdots, n$. Multiplying both sides
of \eqref{lindeca}  by
\[
\left\{\frac{(1- |\xi_t (t)|^2)^{k-2}}{2^{k-2}}\binom{n-1}{k-1}  A(t)\right\}^{-1} Y_j(\theta)
\]
and integrating over $\theta \in \mathbb S^{n-1}$, we obtain
\begin{equation}\label{eq1}
 \phi_j^{''}(t) +\frac{B(t)}{A(t)} \phi_j^{'}(t) +
\left[\frac{ne^{-n\xi(t)}}{e^{-n\xi(t)} +h} (1-\xi_t^2(t))  -\lambda_j \frac{C(t)}{A(t)} \right]  \phi_j (t)
= \widehat r_j(t)
\end{equation}
where
\[
\widehat r_j(t) = \int_{\mathbb S^{n-1}} \widehat r(t,\theta) Y_j(\theta)\,d\theta.
\]
For $j=1,\cdots, n$, $\lambda_j=n-1$, and
$\phi_1^-(t) := e^{-t}(1+\xi^{'}(t))$, $\phi_1^+(t) := e^{t}(1-\xi^{'}(t))$ form a solution basis
to the homogeneous equation
\[
 \phi^{''}(t) +\frac{B(t)}{A(t)} \phi^{'}(t) +
\left[\frac{ne^{-n\xi(t)}}{e^{-n\xi(t)} +h} (1-\xi_t^2(t))  -(n-1)\frac{C(t)}{A(t)} \right]  \phi (t)
=0.
\] 
Since $\phi_j(t)$ is a solution to \eqref{eq1} and $\phi_j(t) \to 0$ as $t\to \infty$, by the
variation of constant formula,
\begin{equation}\label{var1}
\phi_j(t)=c \phi_1^-(t) + \phi_1^-(t) \int_0^t \frac{\phi_1^+(s)  \widehat r_j(s)}{W_1(s)}ds
+ \phi_1^+(t) \int_t^{\infty}  \frac{\phi_1^-(s)  \widehat r_j(s)}{W_1(s)}ds,
\end{equation}
for some constant $c$, where
\[
W_1(s)=\phi_1^+(t) \phi_1^{-'}(t)-\phi_1^-(t)\phi_1^{+'}(t)
\]
is the Wronskian of $\{\phi_1^-(t), \phi_1^+(t)\}$, and satisfies
\[
W_1^{'}(s)=-\frac{B(s)}{A(s)}W_1(s).
\]
Integrating this equation out, using
\[ 
\frac{B(s)}{A(s)}=\left(\frac{2k-n}{k}-\frac{n(k-1)}{k}\frac{e^{-n\xi(s)}}{e^{-n\xi(s)}+h}\right)\xi^{'}(s),
\]
we find
\[
W_1(s)=(\text{const.}) e^{\frac{n-2k}{k}\xi(s)} \left(e^{-n\xi(s)}+h\right)^{-\frac{k-1}{k}}
\]
is a periodic function, having a positive upper and lower bound.
According to our assumption on the decay rate of $r(t,\theta)$, we have
\[
|r_j(s)|\le C e^{-\beta t}.
\]
Thus
\[
\left| \int_t^{\infty}  \frac{\phi_1^-(s)  \widehat r_j(s)}{W_1(s)}ds\right| \lesssim \int_t^{\infty} e^{-(1+\beta)s}ds
 \lesssim e^{-(1+\beta)t},
\]
from which we deduce that
\[
\left| \phi_1^+(t) \int_t^{\infty}  \frac{\phi_1^-(s)  \widehat r_j(s)}{W_1(s)}ds\right| \lesssim
e^{- \beta t}.
\]
When $\beta \ne 1$, we also have

\[
\left|\int_0^t \frac{\phi_1^+(s)  \widehat r_j(s)}{W_1(s)}ds\right| \lesssim
\int_0^t e^{(1-\beta)s}ds \lesssim e^{(1-\beta)t},
\]
from which  we deduce that
\[
\left| \phi_1^-(t) \int_0^t \frac{\phi_1^+(s)  \widehat r_j(s)}{W_1(s)}ds\right| \lesssim
e^{- \beta t}.
\]
Putting these estimates into \eqref{var1}, we have
\[
\left|\phi_j(t)-ce^{-t}(1+\xi^{'}(t))\right|\lesssim e^{- \beta t}.
\]
When $\beta =1$, \eqref{var1} gives the modified estimate.

For $j=0$, $\phi_0^{+}(t):=\xi_h^{'}(t)$ and $\phi_0^{-}(t):= \partial_h \xi_h(t)$ also form a solution basis to
the homogeneous equation
\[
 \phi^{''}(t) +\frac{B(t)}{A(t)} \phi^{'}(t) +
\left[\frac{ne^{-n\xi(t)}}{e^{-n\xi(t)} +h} (1-\xi_t^2(t)) \right]  \phi (t)
=0.
\]
Since $\phi_0(t)$ is a solution to \eqref{eq1} and $\phi_0(t) \to 0$ as $t\to \infty$,
 a variant of \eqref{var1} gives:
\begin{equation}\label{var0}
\phi_0(t)=-\phi_0^-(t) \int_t^{\infty} \frac{\phi_0^+(s)  \widehat r_0(s)}{W_0(s)}ds
+ \phi_0^+(t) \int_t^{\infty}  \frac{\phi_0^-(s)  \widehat r_0(s)}{W_0(s)}ds,
\end{equation}
where
\[
W_0(s)=\phi_0^+(t) \phi_0^{-'}(t)-\phi_0^-(t)\phi_0^{+'}(t)
\]
is the Wronskian of $\{\phi_0^-(t), \phi_0^+(t)\}$, and also satisfies
\[
W_0^{'}(s)=-\frac{B(s)}{A(s)}W_0(s).
\]
Thus, as for $W_1(s)$,  $W_0(s)$ is a periodic function, having a positive upper and lower bound.
Let $T(h)$ denotes the minimal period of the solution $\xi_h(t)$. Then
$\xi_h(t+T(h))=\xi_h(t)$. Differentiating in $h$, we obtain
\begin{equation}\label{perd}
\phi_0^{-}(t+T(h))+ T^{'}(h) \xi^{'}_h(t+T(h))=\phi_0^{-}(t),
\end{equation}
which implies that $\phi_0^{-}(t)$ grows in $t$ at most linearly. Then 
\eqref{var0} would imply that $|\phi_0(t)|\lesssim te^{-\beta t}$.  This is not quite as claimed,
but is good enough to be used in our iterative argument in proving \eqref{goal}. To obtain
the more precise estimate \eqref{app}, note that \eqref{perd} implies that
\[
p(t) :=\phi_0^{-}(t) +\frac{T^{'}(h)}{T(h)} t \xi^{'}_h(t) = \phi_0^{-}(t) +\frac{T^{'}(h)}{T(h)} t \phi_0^+(t)
\]
is $T(h)$ periodic---such behavior can also be deduced from the application of Floquet theory to this case.  
Thus we can express $\phi_0^{-}(t)$ as
$p(t)- \frac{T^{'}(h)}{T(h)} t \phi_0^+(t)$ in \eqref{var0} to obtain
\[
\begin{split}
\phi_0(t)=&-\left(p(t)- \frac{T^{'}(h)}{T(h)} t\phi_0^+(t)\right) 
\int_t^{\infty} \frac{\phi_0^+(s)  \widehat r_0(s)}{W_0(s)}ds
+ \phi_0^+(t) \int_t^{\infty}  \frac{\left(p(s)- \frac{T^{'}(h)}{T(h)} s\phi_0^+(s)\right)  \widehat r_0(s)}{W_0(s)}ds\\
=& - p(t) \int_t^{\infty} \frac{\phi_0^+(s)  \widehat r_0(s)}{W_0(s)}ds 
+ \phi_0^+(t) \int_t^{\infty}  \frac{p(s)\widehat r_0(s)}{W_0(s)}ds\\
& -\frac{T^{'}(h)}{T(h)} \phi_0^+(t) \int_t^{\infty}
\int_s^{\infty} \frac{\phi_0^+(\tau)  \widehat r_0(\tau)}{W_0(\tau)}d\tau ds,
\end{split}
\]
from which follows $|\phi_0(t)|\lesssim e^{-\beta t}$.

Finally, we estimate the decay rate of $\widehat \phi (t,\theta)$. 
This part is analogous to an approach in \cite{TZ06}.
Multiplying both sides of  \eqref{lindeca} by
\[
\left\{\frac{(1- |\xi_t (t)|^2)^{k-2}}{2^{k-2}}\binom{n-1}{k-1}  A(t)\right\}^{-1} \wph,
\]
integrating over $\theta \in \mathbb S^{n-1}$ and using \eqref{ortho1} and \eqref{ortho2}, we find
\begin{equation}\label{hord}
\begin{split}
&\int_{\mathbb S^{n-1}} \left\{ \widehat \phi_{tt} (t,\theta) \wph\ + \frac{B(t)}{A(t)} \widehat \phi_t(t,\theta) 
\wph + 
\frac{ne^{-n\xi(t)}}{e^{-n\xi(t)} +h} (1-\xi_t^2(t)) |\wph|^2 \right\} \,d\theta \\
&- \frac{C(t)}{A(t)} \int_{\mathbb S^{n-1}}|\grad_{\theta} \wph|^2 \,d\theta=
\int_{\mathbb S^{n-1}} \widehat r(t,\theta) \wph \,d\theta,
\end{split}
\end{equation}
where
\[
\widehat r(t,\theta) = \left\{\frac{(1- |\xi_t (t)|^2)^{k-2}}{2^{k-2}}\binom{n-1}{k-1}  A(t)\right\}^{-1} r(t,\theta)
\asymp r(t,\theta).
\]
Defining
\[
y(t)= \sqrt{\int_{\mathbb S^{n-1}} |\wph|^2  \,d\theta},
\]
then 
\[
y^{'}(t)= \int_{\mathbb S^{n-1}} \widehat \phi_t(t,\theta) \wph \,d\theta/y(t), \quad \text{whenever $y(t)>0$,}
\]
and
\[
y(t) y^{''}(t)=\int_{\mathbb S^{n-1}} \left\{\widehat \phi_{tt} (t,\theta) \wph +|\widehat \phi_t(t,\theta)|^2 \right\} \,d\theta 
- |y^{'}(t)|^2.
\]
Cauchy-Schwarz inequality implies that
\[
|y^{'}(t)|^2 \le \int_{\mathbb S^{n-1}} |\widehat \phi_t(t,\theta)|^2 \,d\theta.
\]
Using these relations and 
\[
 \int_{\mathbb S^{n-1}}|\grad_{\theta} \wph|^2  \,d\theta \ge 2n  \int_{\mathbb S^{n-1}}|\wph|^2  \,d\theta
\]
into \eqref{hord}, we obtain
\[
y(t) y^{''}(t) +\frac{B(t)}{A(t)} y(t) y^{'}(t) +
\left[\frac{ne^{-n\xi(t)}}{e^{-n\xi(t)} +h} (1-\xi_t^2(t))  -2n\frac{C(t)}{A(t)} \right] y^2(t) 
\ge -||\widehat r(t,\cdot)||_{L^2(\mathbb S^{n-1})}y(t),
\]
whenever $y(t)>0$, from which we deduce
\begin{equation}\label{hordod}
 y^{''}(t) +\frac{B(t)}{A(t)} y^{'}(t) +
\left[\frac{ne^{-n\xi(t)}}{e^{-n\xi(t)} +h} (1-\xi_t^2(t))  -2n\frac{C(t)}{A(t)} \right] y (t)
\ge -||\widehat r(t,\cdot)||_{L^2(\mathbb S^{n-1})},
\end{equation}
whenever $y(t)>0$.
According to our assumption on $r(t,\theta)$, we have
\[
||\widehat r(t,\cdot)||_{L^2(\mathbb S^{n-1})} \le C e^{-\beta t}
\]
for some constant $C>0$.  By \eqref{odetrans},
\[
\begin{split}
&\left\{\partial_{tt} +\frac{B(t)}{A(t)}\partial_t +\left[ - 2n \frac{C(t)}{A(t)} +
\frac{ne^{-n\xi(t)}}{e^{-n\xi(t)} +h} (1-\xi_t^2(t))\right]\right\} (V^{-1}(t)e^{-\beta t}) \\
&\le \left(\beta^2+E\right)V^{-1}(t) e^{-\beta t}\le -\epsilon  V^{-1}(t) e^{-\beta t}, 
\end{split}
\]
for some $\epsilon>0$ when $\beta < \beta_* $. So $z(t) :=C \epsilon^{-1} (\max V)V^{-1}(t)e^{-\beta t}$
satisfies
\begin{equation}\label{com}
\left\{\partial_{tt} +\frac{B(t)}{A(t)}\partial_t +\left[ - 2n \frac{C(t)}{A(t)} +
\frac{ne^{-n\xi(t)}}{e^{-n\xi(t)} +h} (1-\xi_t^2(t))\right]\right\} (z(t)-y(t)) \le 0,
\end{equation}
whenever $y(t)>0$.
We also know that $y(t)\to 0$ as $t\to \infty$. We may choose $C>0$ large so that $z(0)\ge y(0)$. Then
we claim that $z(t)-y(t) \ge 0$ for all $t\ge 0$, for, if not, $\min (z(t)-y(t)) <0$ is finite, and
is attained at some $t_*$, then $y(t_*)> z(t_*)>0$, so \eqref{com} holds at
$t=t_*$, and  $\partial_t (z(t)-y(t)) |_{t=t_*}=0$, as well as
$\partial_{tt} (z(t)-y(t)) |_{t=t_*} \ge 0$. This contradicts \eqref{com}.
Thus we conclude
\[
 \sqrt{\int_{\mathbb S^{n-1}} |\wph|^2  \,d\theta}= y(t) \le C \epsilon^{-1} (\max V) V^{-1}(t)e^{-\beta t}.
\]
We can now bootstrap this integral estimate to obtain a pointwise decay estimate
\[
|\wph| \lesssim e^{-\beta t}.
\]
When $\beta \ge \beta_*$, we can simply split those components $\phi_j$ of $\phi$ with $\lambda_j=2n$ from $\wph$,
and estimate them as we did for $\phi_j$, $j=0,\cdots, n$, and estimate $\wph$ with an improved
exponential decay rate. 
\end{proof}
We now provide a proof for Theorem~\ref{hodthm}. Our proof is very much like the one in \cite{KMPS} for
the $k=1$ case, once we have obtained the needed linear analysis.
\begin{proof}[Proof of Theorem~\ref{hodthm}]
Our starting point is still
\[
L_{\xi_h(\cdot+\tau)}(\phi) +Q(\phi)+2k c e^{-2k\xi_h(t+\tau)}\phi(t,\theta)=0,
\]
and our premise is:
\begin{equation}
|Q(\phi)| \lesssim e^{-2\alpha t} \quad \text{whenever we have }|\phi,
 \partial \phi, \partial^2 \phi |\lesssim e^{-\alpha t}.
\end{equation}
 We already established

\noindent {\bf Step 1.} For some $\alpha_0 >0$, $|\phi, \partial \phi, \partial^2 \phi|\lesssim e^{-\alpha_0 t}$.

If $\alpha_0 \ge \rho_{n+1}$, we stop and have now proved 
$|w(t,\theta)-\xi_h(t+\tau)|=|\phi(t,\theta)| \lesssim e^{-\rho_{n+1} t}$, where
$ \rho_{n+1} > \sqrt{2}$; if $1< \alpha_0  <  \rho_{n+1}$, we jump to {\bf Step 3};
if $\alpha_0  \le 1$, we move onto

\noindent {\bf Step  2.} Recall that we now have $|Q(\phi)|\lesssim e^{-2 \alpha_0 t}$. 
If  $2 \alpha_0 > \rho_{n+1}$, then we can apply Proposition~\ref{la4} directly to conclude our proof;
If $1< 2 \alpha_0 \le \rho_{n+1}$, then  we certainly still have 
$|Q(\phi)|\lesssim e^{-2 \alpha t}$ for some $1< 2 \alpha < \rho_{n+1}$ and
can apply  Proposition~\ref{la4} to imply that
\begin{equation}\label{tep2}
|w(t,\theta)-\xi_h(t+\tau)- \sum_{j=1}^n a_j e^{-(t+\tau)}(1+\xi^{'}_h(t+\tau))Y_j(\theta)|\lesssim e^{- 2 \alpha t},
\end{equation}
for some constants $a_j$ for $j=1,\cdots, n$, and jump to {\bf Step 3};
if $ 2 \alpha_0 \le 1$, we may take $\alpha_0$ to satisfy $2 \alpha_0 <1$ and apply
 Proposition~\ref{la4} to imply that
\[
|\phi(t,\theta)- \sum_{j=1}^n a_j e^{-(t+\tau)}(1+\xi^{'}_h(t+\tau))Y_j(\theta)|\lesssim  e^{- 2 \alpha_0 t}
\]
for some constants $a_j$ for $j=1,\cdots, n$. This certainly implies that
\begin{equation}\label{start}
|\phi(t,\theta)|\lesssim  e^{- 2 \alpha_0 t}.
\end{equation}
Next we use higher derivative estimates for $w(t,\theta)$ and 
$\xi_h(t+\tau)$ and interpolation with \eqref{start} to obtain
\[
|\phi, \partial \phi, \partial^2 \phi| \lesssim e^{-2\alpha' t}
\]
for any $\alpha'< \alpha_0$.
Now we go back to the beginning
of step 2 and repeat the process with a new $\alpha_1>\alpha_0$ to replace the $\alpha_0$ there, say,
$\alpha_1=1.8 \alpha_0$.
After a finite number of steps, we will reach a stage where $2\alpha >1$ and ready to move onto

\noindent {\bf Step  3.} At this stage, we have $|\phi(t,\theta)|\lesssim e^{-t}$.
Repeating the last part of {\bf Step  2} involving the derivative estimates  for $w(t,\theta)$ and
$\xi_h(t+\tau)$ to bootstrap the estimate for $Q(\phi)$ to $|Q(\phi)| \lesssim e^{-\alpha t} $,
with $\alpha$ can be as close to $2$ as one needs. Then, depending on whether $\rho_{n+1} \ge 2$ or otherwise,
one can apply Proposition~\ref{la4} to obtain \eqref{app} or \eqref{app2}. In the first case, we can continue
the iteration until $2\alpha >2$. But due to the presence of  
$e^{-(t+\tau)}(1+\xi^{'}_h(t+\tau))Y_j(\theta)$ in the estimate for $\phi$, the estimate for
$Q(\phi)$ can not be better than $e^{-2t}$. This explains the appearance of $\text{min}\{2, \rho_{n+1}\}$
in \eqref{goal}.
\end{proof}

\section{Proof of Theorem~\ref{odethm}}
\begin{remark}
First, some comments on the assumptions in Theorem~\ref{odethm}.
\begin{enumerate}
\item[(a).] %\eqref{nond} is some kind of non-degeneracy condition on $H$ at $(0,m)$.
Assumptions \eqref{nond} and \eqref{afi} imply that whenever $\beta^{'}(\tau_j)=0$ and $|\beta(\tau_j)-m|<
\epsilon_1$, then $|\beta(\tau_j)-m|$ is in fact bounded above by $e_2(\tau_j)^{1/l}$.
\item[(b).] In the case that $\psi$ is   non-constant, it follows that
\[
|f(0,m)| = a>0.
\]
Thus there exists $0<\epsilon_2 \le \epsilon_1$ such that
\begin{equation}\label{flow}
|f(x,y)| \ge 3a/4, \quad \text{for all $(x,y)$ with $|x|+|y-m| < \epsilon_2$.}
\end{equation}
Let $T_0$ be such that $|e_1(t)|< a/4$ for $t \ge T_0$. Then
\begin{equation}\label{flowb}
|\beta^{''}(t)|\ge a/2,  \quad \text{whenever $|\beta^{'}(t)|+|\beta(t)-m|  < \epsilon_2$ and $t\ge T_0$.}
\end{equation}
\item[(c).] By linearization at $\psi (t_{**}+\cdot)$, there exists $B>0$ depending on $f$,  $T$ and
the upper bound of $|\beta(\cdot)|+|\beta^{'}(\cdot)|$  such
that
\begin{equation}\label{lin}
\begin{split}
&|\beta(t_{*}+\tau)-\psi(t_{**} +\tau)|+|\beta^{'}(t_{*}+\tau)-\psi^{'}(t_{**} +\tau)|\\
\le& B \left(|\beta(t_{*})-\psi(t_{**}) |+|\beta^{'}(t_{*})-\psi^{'}(t_{**})|+\max_{|t-t_{*}|\le 2T} 
|e_1(t)| \right)
\end{split}
\end{equation}
for all $-2T\le \tau \le 2T$. One ingredient of our proof in case (ii) is to find a sequence of $\tau_j \to \infty$
with $\tau_{j+1}-\tau_j\approx T$ as $j\to \infty$, such that $\beta^{'}(\tau_j)=0$ and $|\beta(\tau_j)-m| 
\to 0$ as $j\to \infty$. Then applying \eqref{lin} to $\beta(\tau_j+\tau)$ and $\psi(\tau)$
would imply that
\[
|\beta(\tau_j+\tau)-\psi(\tau)|+|\beta^{'}(\tau_j+\tau)-\psi^{'}(\tau)|
\le  B \left(|\beta(\tau_j)-m|+ \max_{0\le \tau\le 2T}|e_1(\tau_j+\tau)|\right),
\]
for $0\le \tau \le 2T$.
\end{enumerate}
\end{remark}
Here is an outline of the main steps in our proof: we first use \eqref{appro2}, \eqref{lin} and 
\eqref{flowb} to deduce that $\beta (t)$ will have critical point for large $t$ with its critical value close to $m$;
then use (a) to prove that the difference between the critical value of $\beta (t)$ with $m$ is actually bounded above by
$e_2(t)^{1/l}$; then we can iterate this argument indefinitely and account for the possible time shift between 
consecutive times that  $\beta (t)$ attains a critical value near  $m$.
We can now put the ingredients together to provide a complete proof.
\begin{proof} First, for some $1/2>\kappa>0$ to be determined, 
by \eqref{appro2} and \eqref{lin}  there exists $t_{i_0}>T_0$ such that
\begin{equation}\label{lin2}
|\beta(t_{i_0}+\tau)-\psi(-s +\tau)|+|\beta^{'}(t_{i_0}+\tau)-\psi^{'}(-s +\tau)|
< \kappa \epsilon_2 \text{    for $|\tau |\le 2T$.} 
\end{equation}
First, we will dispose of case (i): when $\psi (t) \equiv m$ is a constant,
\eqref{nond1} and \eqref{afi} imply that
\begin{equation}\label{case 1}
|\beta^{'}(t)|^l+|\beta(t)-m|^l\le A^{-1} e_2(t),
\end{equation}
as long as $|\beta^{'}(t)|+ |\beta(t)-m| \le \epsilon_1$. This, together with \eqref{lin2}, apparently implies
that \eqref{case 1} continues to hold for all $t\ge T_0$.
So we are left to deal with the case that $\psi (t)$ is non-constant.
Noting that $\psi(0)=m$ and $\psi^{'}(0)=0$, we have
by \eqref{lin2} that
\[
|\beta(t_{i_0}+s)-m|+ |\beta^{'}(t_{i_0}+s)| < \kappa \epsilon_2.
\]
We now prove that there exists $\delta_0$ with $|\delta_0|\le 2a^{-1}|\beta^{'}(t_{i_0}+s)|$ such that
\begin{equation}\label{ind0}
\beta^{'}(t_{i_0}+s +\delta_0)=0.
\end{equation}
Let $\Lambda= \sup_{\mathbb R}\{|\psi^{''}(t)|, |\psi^{'}(t)|\}$. Then
\begin{equation}\label{lin3}
|\psi (\tau )-m| +|\psi^{'}(\tau)| \le  2 \Lambda |\tau| < \epsilon_2/2,
\end{equation}
for $|\tau|< \epsilon_2/(4\Lambda)$. 
 Together with \eqref{lin2} and \eqref{lin3}, we know
\[
|\beta(t_{i_0}+s+\tau)-m|+ |\beta^{'}(t_{i_0}+s+\tau)| < \epsilon_2, \quad \text{for 
$|\tau|< \epsilon_2/(4\Lambda)$}.
\] 
Thus, by  \eqref{flowb}, 
we have $|\beta^{''}(t_{i_0}+s+\tau)| \ge a/2$ for $|\tau|< \epsilon_2/(4\Lambda)$. Then
by elementary calculus, there exists $\delta_0$ such that $\beta^{'}(t_{i_0}+s+\delta_0)=0$ and
\begin{equation}
|\delta_0| \le 2a^{-1} |\beta^{'}(t_{i_0}+s)|\le 2a^{-1} \kappa  \epsilon_2 
< \epsilon_2/(4\Lambda),
\end{equation}
provided $\kappa$ is chosen to satisfy the last inequality above. We fix such a $\kappa$ now. 
Set $\tau_0= t_{i_0}+s+ \delta_0$. Note that we have $|\beta (\tau_0)-m|< \epsilon_2$. Thus
from assumption \eqref{nond}, we have
\[
\begin{split}
A|\beta(\tau_0)-m|^l &= A|\beta(\tau_0)-\psi(0)|^l\\
&\le |H(0, \beta(\tau_0))- H(0, \psi(0))| \\
&= |H(\beta^{'}(\tau_0), \beta(\tau_0))- H(\psi^{'}(0), \psi(0))|\\
&= |H(\beta^{'}(\tau_0), \beta(\tau_0))-0|\\
&\le e_2(\tau_0),
\end{split}
\]
which implies that
\begin{equation}\label{ind0b}
|\beta(\tau_0)-m| \le \left(\frac{ e_2(\tau_0)}{A}\right)^{1/l}.
\end{equation}
Next we apply \eqref{lin} to $\beta(\tau_0+\tau)$ and $\psi (\tau)$ to obtain
\begin{equation}\label{ind1}
\begin{split}
&|\beta(\tau_0+\tau)-\psi (\tau)|+ |\beta^{'}(\tau_0+\tau)-\psi^{'} (\tau)| \\
\le& B \left( |\beta(\tau_0)-m| + \max_{\tau_0\le t \le \tau_0+2T} |e_1(t)| \right) \\
\le&  B \left( \left(\frac{ e_2(\tau_0)}{A}\right)^{1/l} + \max_{\tau_0\le t \le \tau_0+2T} |e_1(t)| \right)\\
\end{split}
\end{equation}
for $0\le \tau \le 2T$.  Repeating the above argument, and in choosing $T_0$ also make sure that
\[
B \left( \left(\frac{ e_2(\tau)}{A}\right)^{1/l} + \max_{\tau\le t \le \tau+2T} |e_1(t)| \right)
< \kappa \epsilon_2
\]
for $\tau \ge T_0$, we obtain
$\delta_1$ such that $\tau_1 = \tau_0+T+\delta_1$ satisfies
\begin{gather}\label{ind2}
\beta^{'}(\tau_1)=0, \\
|\delta_1| \le 2a^{-1} |\beta^{'}( \tau_0+T)|\le  2a^{-1} 
 B \left( \left(\frac{ e_2(\tau_0)}{A}\right)^{1/l} + \max_{\tau_0\le t \le \tau_0+2T} |e_1(t)| \right), \\
|\beta(\tau_1)-m|  \le \left(\frac{ e_2(\tau_1)}{A}\right)^{1/l}.
\end{gather}
We can now inductively find $\tau_{j}= \tau_{j-1}+T+\delta_j$ such that
\begin{gather}
\beta^{'}(\tau_j)=0, \\
|\delta_j| \le 2a^{-1} |\beta^{'}( \tau_{j-1}+T)|\le  2a^{-1} 
 B \left( \left(\frac{ e_2(\tau_{j-1})}{A}\right)^{1/l} + \max_{\tau_{j-1}\le t \le \tau_{j-1}+2T} 
|e_1(t)| \right), \label{del} \\
|\beta(\tau_j)-m|  \le \left(\frac{ e_2(\tau_j)}{A}\right)^{1/l},\\
\begin{split}\label{ind3}
&|\beta(\tau_j+\tau)-\psi (\tau)|+ |\beta^{'}(\tau_j+\tau)-\psi^{'} (\tau)| \\
\le& B \left( |\beta(\tau_j)-m| + \max_{\tau_j\le t \le \tau_j+2T} |e_1(t)| \right) \\
\le&  B \left( \left(\frac{ e_2(\tau_j)}{A}\right)^{1/l} + \max_{\tau_j\le t \le \tau_j+2T} |e_1(t)| \right)\\
\end{split}
\end{gather}
for $0\le \tau \le 2T$.
Set $s_j = \tau_j - jT$. Then $s_j=s_{j-1}+\delta_j$, and due to estimate \eqref{del} and assumption \eqref{deca}
$s_{\infty}=\lim_{j \to \infty} s_j$ exists and
equals $s_0 + \sum_{j=1}^{\infty}\delta_j$. \eqref{ind3} can be rewritten, with $t=\tau_j+\tau$, as
\[
\begin{split}
&|\beta(t)-\psi (t-s_j)|+|\beta^{'}(t)-\psi^{'} (t-s_j)| \\
=&|\beta(t)-\psi (t-\tau_j)|+|\beta^{'}(t)-\psi^{'} (t-\tau_j)| \\
\le & B \left( \left(\frac{ e_2(\tau_j)}{A}\right)^{1/l} + \max_{\tau_j\le t \le \tau_{j+1}} |e_1(t)| \right)\\
\le & B \int_{\tau_j-1}^{\infty} \left( \left(\frac{ e_2(t^{'})}{A}\right)^{1/l} + 
\max_{ t^{'} \le \tau } |e_1(\tau)| \right) dt^{'}
\end{split}
\]
for $ \tau_j \le t \le \tau_{j+1}$, which further implies that
\begin{equation}\label{fe}
\begin{split}
&|\beta(t)-\psi (t-s_{\infty})|+|\beta^{'}(t)-\psi^{'} (t-s_{\infty})|\\
\le & |\beta(t)-\psi (t-s_j)|+|\beta^{'}(t)-\psi^{'} (t-s_j)| +
      |\psi (t-s_{\infty}) -\psi (t-s_j)| +|\psi^{'} (t-s_{\infty})-\psi^{'} (t-s_j)| \\
\le & B \left( \left(\frac{ e_2(\tau_j)}{A}\right)^{1/l} + \max_{\tau_j\le t \le \tau_{j+1}} |e_1(t)| \right) +
\Lambda |s_{\infty}-s_j| \\
\le &  B \left( \left(\frac{ e_2(\tau_j)}{A}\right)^{1/l} + \max_{\tau_j\le t \le \tau_{j+1}} |e_1(t)| \right) +
 \Lambda \sum_{k=j+1}^{\infty} |\delta_k|\\
\le & C \int_{\tau_j-1}^{\infty} \left( \left(\frac{ e_2(t^{'})}{A}\right)^{1/l} + 
\max_{t^{'}  \le \tau } |e_1(\tau)| \right) dt^{'}
\end{split}
\end{equation}
for some constant $C>0$ and $ \tau_j \le t \le \tau_{j+1}$.
\eqref{del} and our assumption \eqref{deca} imply that the expression on the right side of the above
inequality tends to $0$ as $j\to \infty$, thus proving \eqref{conc}.
%Comparing this with \eqref{appro2}, we conclude that $s_{\infty}=s + m T$ for some $m \in \mathbb Z$, thus
%proving \eqref{conc}.
\end{proof}

\section{Appendix}

\begin{proof}[Proof of Lemma~\ref{la3}]
Introduce a new variable $\psi (t) =V(t)\phi(t)$ for some $V(t)$ to be chosen. Then
\[
\begin{split}
&V(t) {L}_j[\phi]\\
= &\psi_{tt}(t)+\left\{ \left[1-(n-1)\frac{C(t)}{A(t)}\right]\xi_t(t)-2\frac{V^{'}(t)}{V(t)}\right\} \psi_t(t) \\
&+ \left\{2 \frac{|V^{'}(t)|^2}{V(t)^2}-\frac{V^{''}(t)}{V(t)}- \frac{V^{'}(t)}{V(t)} \left[1-(n-1)\frac{C(t)}{A(t)}\right]\xi_t(t)
-\lambda_j \frac{C(t)}{A(t)} +\frac{ne^{-n\xi(t)}}{e^{-n\xi(t)} +h}(1-\xi_t^2(t))  \right\} \psi(t).
\end{split}
\]
Choose $V(t)$ such that
$\displaystyle{
 \left[1-(n-1)\frac{C(t)}{A(t)}\right]\xi_t(t)-2\frac{V^{'}(t)}{V(t)}=0.
}$
This amounts to
\[
\begin{split}
\left(2 \ln V(t)\right)_t=& \left[1-\frac{n-k}{k}-\frac{n(k-1)e^{-n\xi(t)}}{ k(e^{-n\xi(t)} +h)}\right]\xi_t(t) \\
=& \left[ (2-\frac nk) \xi(t)+\frac{k-1}{k} \ln \left( e^{-n\xi(t)} +h\right) \right]_t.
\end{split}
\]
Thus we can take
$\displaystyle{
V(t)= e^{(1-\frac{n}{2k})\xi(t)}\left( e^{-n\xi(t)} +h\right)^{\frac{k-1}{2k}}.
}$
Then
\[
\begin{split}
&V(t) {L}_j[\phi]\\
= &\psi_{tt}(t)+ \left\{- \frac{V^{''}(t)}{V(t)} 
-\lambda_j \frac{C(t)}{A(t)} +\frac{ne^{-n\xi(t)}}{e^{-n\xi(t)} +h}(1-\xi_t^2(t))  \right\} \psi(t)\\
=& \psi_{tt}(t)+ E(t) \psi(t),
\end{split}
\]
where
\[
E(t) = \left\{- \frac{V^{''}(t)}{V(t)} 
-\lambda_j \frac{C(t)}{A(t)} +\frac{ne^{-n\xi(t)}}{e^{-n\xi(t)} +h}(1-\xi_t^2(t))  \right\},
\]
and
\[
\begin{split}
\frac{V^{''}(t)}{V(t)}  =& \left[1-\frac{n}{2k} -
\frac{n(k-1)  e^{-n\xi(t)}}{2k (e^{-n\xi(t)} +h)}\right] \xi^{''} \\
&+ \frac{n^2(k-1)h e^{-n\xi}  |\xi^{'}|^2}{2k(e^{-n\xi(t)} +h)^2 }
 + \left[ 1-\frac{n}{2k} -\frac{n(k-1) e^{-n\xi(t)}}{2k (e^{-n\xi(t)} +h)}\right]^2|\xi^{'}|^2.
\end{split}
\]
Using
\[
\xi_{tt}(t)=\frac{n}{2k}e^{-2k\xi(t)}(1-\xi_t^2(t))^{1-k} -\frac{n-2k}{2k}(1-\xi_t^2(t)),
\]
and
\[
e^{(2k-n)\xi(t)}(1-\xi_t^2(t))^k= e^{-n\xi(t)} +h,
\]
we have
\[
\begin{split}
\frac{V^{''}(t)}{V(t)}  =& \left[ \frac{(2k-n)^2}{4k^2}+\frac{n(n-2k)(k-2)  e^{-n\xi(t)}}{ 4k^2(e^{-n\xi(t)} +h)}
- \frac{n^2(k-1)}{4k^2}\left(\frac{ e^{-n\xi(t)}}{e^{-n\xi(t)} +h}\right)^2\right](1-|\xi^{'}|^2) \\
&+ \frac{n^2(k-1)h e^{-n\xi}  |\xi^{'}|^2}{2k(e^{-n\xi(t)} +h)^2 }+
 \left[ 1-\frac{n}{2k} -\frac{n(k-1) e^{-n\xi(t)}}{2k (e^{-n\xi(t)} +h)}\right]^2|\xi^{'}|^2\\
=& \left[ \frac{(2k-n)^2}{4k^2}+\frac{n(n-2k)(k-2)  e^{-n\xi(t)}}{ 4k^2(e^{-n\xi(t)} +h)}
- \frac{n^2(k-1)}{4k^2}\left(\frac{ e^{-n\xi(t)}}{e^{-n\xi(t)} +h}\right)^2\right] \\
&+\left\{- \frac{(2k-n)^2}{4k^2}- \frac{n(n-2k)(k-2)  e^{-n\xi(t)}}{ 4k^2(e^{-n\xi(t)} +h)}
+ \frac{n^2(k-1)}{4k^2}\left(\frac{ e^{-n\xi(t)}}{e^{-n\xi(t)} +h}\right)^2  \right.\\
& \left. + \frac{n^2(k-1)h e^{-n\xi}  }{2k(e^{-n\xi(t)} +h)^2 } +
\left[ 1-\frac{n}{2k} -\frac{n(k-1) e^{-n\xi(t)}}{2k (e^{-n\xi(t)} +h)}\right]^2\right\} |\xi^{'}|^2\\
=& \left[ \frac{(2k-n)^2}{4k^2}+\frac{n(n-2k)(k-2)  e^{-n\xi(t)}}{ 4k^2(e^{-n\xi(t)} +h)}
- \frac{n^2(k-1)}{4k^2}\left(\frac{ e^{-n\xi(t)}}{e^{-n\xi(t)} +h}\right)^2\right] \\
&+\left[\frac{n^2(k-1)h e^{-n\xi}  }{2k(e^{-n\xi(t)} +h)^2 }+\frac{n(n-2k)  e^{-n\xi(t)}}{4k(e^{-n\xi(t)} +h)} + 
\frac{n^2(k-1)}{4k}\left(\frac{ e^{-n\xi(t)}}{e^{-n\xi(t)} +h}\right)^2\right] |\xi^{'}|^2
\end{split}
\]
Since we are interested in getting an upper bound for $E(t)$, and it's not easy to find more useful
bound for terms of the form (negative factor)$|\xi^{'}|^2$, so will drop such terms in our estimates
and obtain, in the case $2k\le n$ and $\lambda_j\ge 2n$,
\[
\begin{split}
E(t) \le & -  \frac{(2k-n)^2}{4k^2}- \frac{n(n-2k)(k-2)  e^{-n\xi(t)}}{ 4k^2(e^{-n\xi(t)} +h)}
+ \frac{n^2(k-1)}{4k^2}\left(\frac{ e^{-n\xi(t)}}{e^{-n\xi(t)} +h}\right)^2\\
& -\frac{2n(n-k)}{k(n-1)} -\frac{2n^2(k-1)e^{-n\xi(t)}}{k(n-1)(e^{-n\xi(t)} +h)}\\
\le & -  \frac{(2k-n)^2}{4k^2}  -\frac{2n(n-k)}{k(n-1)} \\
&- \left\{  \frac{n(n-2k)(k-2)}{4k^2}- \frac{n^2(k-1)}{4k^2}+ \frac{2n^2(k-1)}{k(n-1)}\right\}
\frac{ e^{-n\xi(t)}}{e^{-n\xi(t)} +h} \\
=  & -  \frac{(2k-n)^2}{4k^2}  -\frac{2n(n-k)}{k(n-1)} 
- \frac{2(n+3)k^2-4(n+1)k -n(n-1)}{4k^2(n-1)} \frac{n e^{-n\xi(t)}}{e^{-n\xi(t)} +h}.
\end{split}
\]
When $2(n+3)k^2-4(n+1)k -n(n-1)\ge 0$, we obtain
\[
E(t) \le  -  \frac{(2k-n)^2}{4k^2}  -\frac{2n(n-k)}{k(n-1)} = -\left(\frac{n}{2k}-1\right)^2-
\frac{2n}{n-1}\left(\frac nk -1\right) \le -2 -\frac{2}{n-1},
\]
provided $2k\le n$; while if $2(n+3)k^2-4(n+1)k -n(n-1)\le 0$, we obtain
\[
\begin{split}
E(t) &\le  -  \frac{(2k-n)^2}{4k^2}  -\frac{2n(n-k)}{k(n-1)} +\frac{-2n(n+3)k^2+4n(n+1)k+n^2(n-1)}{4k^2(n-1)}\\
&= -\frac{n+1}{2}\,!
\end{split}
\]
In all cases we conclude the proof of Lemma~\ref{la3}.
\end{proof}


\begin{thebibliography}{99}
\bibitem{A5?}
A.D. Alexandrov, \emph{Uniqueness theorems for surfaces in the large I--V}, Vestnik Leningrad Univ. \textbf{11} \#19,
5--17 (1956);  \textbf{12} \#7, 15--44 (1957);  \textbf{13} \#7, 14--26 (1958); \textbf{13} \#13, 27--34 (1958); 
\textbf{13} \#19, 5–8 (1958);
English transl. in Am. Math. Soc. Transl. \textbf{21}, 341--354, 354--388, 389--403, 403--411, 412--416 (1962).
\bibitem{Av}
P. Aviles, \emph{A study of the singularities of solutions of a class of nonlinear elliptic
partial differential equations}, Comm. PDE \textbf{7} (1982), 609-643.
\bibitem{BV}
S. Brendle and J. Viaclovsky, \emph{A variational characterization for $\sigma_{n/2}$}, 
Calc. Var. PDE. \textbf{20} (2004), no.4, 399-402.
\bibitem{CGS}
Luis A. Caffarelli, B. Gidas and J. Spruck,  \emph{Asymptotic symmetry and local behavior of semilinear elliptic
equations with critical Sobolev growth}, Comm. Pure Appl. Math.  \textbf{42} (1989), 271-297.
\bibitem{CGY1} 
S.-Y. A. Chang, M. Gursky and P. Yang, \emph{An equation of 
Monge-Amp\`ere type in conformal geometry, and four-manifolds of positive
Ricci curvature}, Annals of Math., \textbf{155} (2002), 709-787.
\bibitem{CGY2} 
S.-Y. A. Chang, M. Gursky and P. Yang, \emph{An a priori estimate for a fully nonlinear
equation on four-manifolds}, Dedicated to the memory of Thomas H. Wolff. 
 J. Anal. Math.  \textbf{87}  (2002), 151--186. 
\bibitem{CGY3} 
S.-Y. A. Chang, M. Gursky and P. Yang, \emph{Entire solutions of a fully nonlinear equation}, 
``Lectures in Partial Differential Equations in honor of Louis Nirenberg's 
75th birthday", chapter 3. International Press, 2003.
\bibitem{CHY05}S.-Y. A. Chang, Z. Han, and P. Yang, \emph{Classification of singular radial solutions to the
$\sigma_k$ Yamabe equation on annular domains}, J. Diff. Eqn. \textbf{216} (2005), 482-501.
\bibitem{CHY2}
S.-Y. A. Chang, Z. Han, and P. Yang, \emph{Apriori estimates for solutions of the prescribed  
$\sigma_2$ curvature equation on $S^4$}, manuscript, 2004, see also \cite{CHY09}.
\bibitem{CHY09}
S.-Y. A. Chang, Z. Han, and P. Yang, \emph{On the prescribing $\sigma_2$ curvature equation on $\mathbb S^4$},
preprint, 2009, 39 pages. arXiv:0911.0375v1[math.DG].
%\bibitem{CHY3}
%S.-Y. A. Chang, Z. Han, and P. Yang, \emph{Injectivity criteria for the developing map
%of a locally conformally flat manifold}, work in progress.
\bibitem{CHgY}
S.-Y. A. Chang, F. Hang, and P. Yang, \emph{On a class of locally conformally flat
manifolds}, IMRN 2004, No. 4, 185-209.
\bibitem{CQY}
S.-Y. A. Chang, J. Qing, and P. Yang, \emph{Compactification of a class of conformally flat manifold}, Invent. Math.
\textbf{142} (2000), 65-93.
\bibitem{CL95}
C.-C. Chen and C.-S. Lin, \emph{Local behavior of singular positive solutions of semilinear 
elliptic equations with Sobolev exponent},  Duke Math. J.  78  (1995),  no. 2, 315--334.
\bibitem{CL99}
C.-C. Chen and C.-S. Lin, \emph{On the asymptotic symmetry of singular solutions of the scalar curvature equations},
  Math. Ann.  \textbf{313}  (1999),  no. 2, 229--245. 
\bibitem{CL}
E. Coddington and N. Levison, \emph{Theory of Ordinary  Differential Equations}, McGraw-Hill, 1955.
\bibitem{Fin}
D. Finn, \emph{Positive Solutions to Nonlinear Elliptic Equations with
Prescribed Singularities}, Ph.D thesis, Northeastern University, 1995.
\bibitem{GNN}
B. Gidas, W.-M. Ni, and L. Nirenberg, \emph{Symmetry and related properties via the maximum principle},
  Comm. Math. Phys. \textbf{68}  (1979), no. 3, 209--243.
\bibitem{G}
M. Gonzalez, \emph{Singular sets of a class of fully non-linear equations in conformal geometry},
Ph.D thesis, Princeton University, 2004.
\bibitem{G2}
M. Gonz\'alez, \emph{Singular sets of a class of locally conformally flat
manifolds}, Duke Math. J. \textbf{129}  (2005),  no. 3, 551--572. 
\bibitem{MG}
M. Gonz\'alez, \emph{Removability of singularities for a class of
fully non-linear equations}, Calculus of Variations and Partial Differential Equations 
\textbf{27}  (2006),  no. 4, 439--466.
\bibitem{GLW}
P. Guan, C.S. Lin and G. Wang, \emph{Schouten tensor and some topological properties.}
  Comm. Anal. Geom.  \textbf{13}  (2005),  no. 5, 887--902. 
\bibitem{GVW}
P. Guan, J. Viaclovsky and G. Wang, \emph{Some properties of the Schouten tensor 
and applications to conformal geometry},
Transactions of American Math. Society, \textbf{355} (2003), 925-933.
\bibitem{GW1}
P. Guan and G. Wang, \emph{Local estimates for a class of fully nonlinear 
equations arising from conformal geometry},
International Mathematics Research Notices, V. 2003, Issue 26(2003), 1413-1432.
\bibitem{GW2}
P. Guan and G. Wang, \emph{A fully nonlinear conformal flow on 
locally conformally flat manifolds},
Journal fur die reine und angewandte Mathematik, \textbf{ 557} (2003), 219-238.
\bibitem{G1}
M. Gursky, \emph{The principal eigenvalue of a conformally invariant differential 
operator, with an application to semilinear elliptic PDE}, Comm. Math. Phys. 
\textbf{207} (1999), 131--143.
\bibitem{GV1}
M. Gursky and J. Viacolvsky, \emph{ A new variational characterization of 
three-dimensional space forms}, Invent. Math. \textbf{145} (2001), 251--278.
\bibitem{GV2} 
M. Gursky and J. Viacolvsky, \emph{Fully nonlinear equations on Riemannian manifolds
 with negative curvature}, Indiana Univ. Math. J. \textbf{52} (2003), no. 2, 399-420. 
\bibitem{GV3}
M. Gursky and J. Viacolvsky, \emph{A fully nonlinear equation on four-manifolds with 
positive scalar curvature}, Journal of Differential Geometry \textbf{63} (2003), no.1, 131--154.
\bibitem{GV}
M. Gursky and J. Viacolvsky, \emph{Convexity and singularities of curvature equations in conformal geometry},
  Int. Math. Res. Not.  (2006), Art. ID 96890, 43 pp. 
\bibitem{H}
Z. Han, \emph{Local Pointwise Estimates for Solutions of the $\sigma_2$ Curvature Equation on $4$-Manifolds},
IMRN 2004, no. 79, 4269-4292.
\bibitem{H06}
Z. Han, \emph{A Kazdan-Warner type identity for the $\sigma\sb k$ curvature},
  C. R. Math. Acad. Sci. Paris  \textbf{342}  (2006),  no. 7, 475--478.
\bibitem{KMPS}
N. Korevaar, R. Mazzeo, F. Pacard and R. Schoen, \emph{Refined asymptotics for constant scalar curvature metrics
with isolated singularities}, Invent. Math. \textbf{135} (1999), 233-272.
\bibitem{Lab03}
D. Labutin, \emph{Wiener regularity for large solutions of nonlinear equations},
Ark. Mat. \textbf{41} (2003), no. 2, 307--339.
\bibitem{leG}
J.F. Le Gall, \emph{Spatial Branching Processes, Random Snakes and Partial Differential Equations},
 Lectures in Mathematics ETH Zürich, Birkhäuser, Basel, 1999.
\bibitem{LL03} 
Aobing Li and YanYan Li, \emph{On some conformally invariant fully nonlinear equations},
 Comm. Pure Appl. Math. \textbf{56} (2003), 1414-1464.
\bibitem{LL05}
Aobing Li and YanYan Li, \emph{On some conformally invariant fully nonlinear equations, Part II:
Liouville, Harnack and Yamabe},  Acta Math.  \textbf{195}  (2005), 117--154.
\bibitem{CL96}
Congming Li, \emph{Local asymptotic symmetry of singular solutions
to nonlinear elliptic equations}, Invent. math. \textbf{123}, 221-231 (1996).
\bibitem{Li06}
 YanYan Li, \emph{Conformally invariant fully nonlinear elliptic equations and isolated singularities},
Journ. Funct. Anal. \textbf{233} (2006), 380-425.
\bibitem{LN}
C. Loewner and L. Nirenberg, \emph{Partial Differential Equations invariant
under conformal and projective transformation}, Contributions to Analysis, Academic
Press,245-275, 1975.
\bibitem{MV09}
M. Marcus, and L. Veron, \emph{Boundary trace of positive solutions of semilinear elliptic 
equations in Lipschitz domains}, arXiv:0907.1006, [math.AP], 16 Jul 2009.
\bibitem{MP1}
R. Mazzeo and F. Pacard, \emph{A  construction of singular solutions  for 
a semilinear elliptic equation using  asymptotic analysis}, J. Diff. Geom.
\textbf{44} (1996), no.2, 331-370.
\bibitem{MP2}
R. Mazzeo and F. Pacard, \emph{Constant scalar curvature metrics with isolated
singularities},  Duke Math. J. \textbf{99}  (1999),  no. 3, 353--418.
\bibitem{MS}
R. Mazzeo and N. Smale, \emph{ Conformally flat metrics of constant
positive scalar curvature on subdomains of the sphere}, J. Diff. Geom.
\textbf{34} (1988), 581-621.
\bibitem{MPU}
R. Mazzeo, D. Pollack and K. Uhlenbeck, \emph{Moduli spaces of singular
Yamabe metrics}, J. Amer. Math. Soc. \textbf{9} (1996), No. 2, 303-344.
\bibitem{Mc}
R. McOwen, \emph{Singularities and the conformal scalar curvature equation}, 
Geometric Analysis and Nonlinear Partial Differential Equations (Denton, Tex., 1990), Lecture
Notes in Pure and Appl. Math. \textbf{144}, Dekker, New York, 1993, 221-233.
\bibitem{P}
F. Pacard, \emph{The Yamabe problem on subdomains of even-dimensional spheres}, Topol.
Methods Nonlinear Anal. \textbf{6} (1995), 137-150.
\bibitem{S88}
R. Schoen, \emph{The existence of weak solutions with prescribed singular behavior
for a conformally invariant scalar equation}, Comm. Pure Appl. Math. \textbf{41} (1988), 317-392.
\bibitem{SY}
R. Schoen and S.T. Yau, \emph{Conformally flat manifolds, Klein groups and
scalar curvature}, Invent. Math. \textbf{92}(1988), 47-72.
\bibitem{S71}
J. Serrin, \emph{A symmetry problem in potential theory}, Arch. Ration. Mech.  \textbf{43} (1971), 304-318.
\bibitem{TZ06}
Steven D. Taliaferro and Lei Zhang, \emph{Asymptotic symmetries for conformal scalar curvature 
equations with singularity}, Calc. Var. Partial Differential Equations \textbf{26} (2006),  401-428.
%\bibitem[TW1]{TW1}
%N. Trudinger and X. J. Wang, \emph{Hessian measures. I}, Dedicated to Olga Ladyzhenskaya.  
%Topol. Methods Nonlinear Anal.  \textbf{10}  (1997),  no. 2, 225--239. 
%\bibitem[TW2]{TW2}
%N. Trudinger and X. J. Wang, \emph{Hessian measures. II},
%  Ann. of Math. (2)  \textbf{150}  (1999),  no. 2, 579--604. 
%\bibitem[TW3]{TW3}
%N. Trudinger and X. J. Wang, \emph{
%Hessian measures. III},  J. Funct. Anal.  \textbf{193}  (2002),  no. 1, 1--23. 
\bibitem{Ver}
L. Veron, \emph{Singularit\'es \'eliminables d'\'equations elliptiques non
lin\'eaires}, J. Diff. Eq., \textbf{41} (1981), 225-242.
\bibitem{Via1}
J. Viaclovsky, \emph{Conformal geometry, contact geometry, and the calculus of 
variations}, Duke Math. J. \textbf{101} (2000), 283-316.
\bibitem{V00}
J. Viaclovsky, \emph{Some fully nonlinear equations in conformal geometry},
 Differential equations and mathematical physics (Birmingham, AL, 1999),  425--433,
 AMS/IP Stud. Adv. Math., 16, Amer. Math. Soc., Providence, RI, 2000.
\bibitem{Via2}
J. Viaclovsky, \emph{Estimates and existence results for some fully nonlinear elliptic
equations on Riemannian manifolds},  Communications in Analysis and Geometry 
\textbf{10} (2002), no.4, 815-846. 
\bibitem{Via3}
J. Viaclovsky, \emph{Conformally Invariant Monge-Amp\`ere Partial Differential Equations:
Global Solutions}, Trans. Amer. Math. Soc. \textbf{352} (2000), no. 9, 4371--4379. 
\end{thebibliography}
\end{document}